\documentclass[11pt]{article}
\usepackage{amsmath,amssymb,amsthm,xypic,mathrsfs} 
\input xy
\xyoption{all}

\usepackage{hyperref}

\usepackage[totalheight=21.6 true cm, totalwidth=15.8 true cm]{geometry}
\usepackage{color}

\title{Finiteness properties of affine difference algebraic groups}

\author{Michael Wibmer\thanks{The author was supported by the NSF grants DMS-1760212, DMS-1760413, DMS-1760448 and the Lise Meitner grant M-2582-N32 of the Austrian Science Fund FWF.}}

\newtheorem{theo}{Theorem}[section]
\newtheorem{lemma}[theo]{Lemma}
\newtheorem{prop}[theo]{Proposition}
\newtheorem{cor}[theo]{Corollary}
\newtheorem{defi}[theo]{Definition}
\newtheorem{rem}[theo]{Remark}

\newtheorem{conj}[theo]{Conjecture}
\newtheorem*{theoremA}{Theorem A}
\newtheorem*{conjB}{Conjecture B}

\theoremstyle{definition}
\newtheorem{ex}[theo]{Example}

\newcommand{\f}{\phi}

\newcommand{\ida}{\mathfrak{a}}
\newcommand{\idb}{\mathfrak{b}}

\newcommand{\p}{\mathfrak{p}}
\newcommand{\q}{\mathfrak{q}}
\newcommand{\m}{\mathfrak{m}}

\newcommand{\spec}{\operatorname{Spec}}

\newcommand{\G}{\mathcal{G}}

\newcommand{\Gl}{\operatorname{GL}}

\newcommand{\Hom}{\operatorname{Hom}}

\newcommand{\trdeg}{\operatorname{trdeg}}
\newcommand{\strdeg}{\sigma\text{-}\operatorname{trdeg}}

\newcommand{\V}{\mathbb{V}}

\newcommand{\id}{\operatorname{id}}

\newcommand{\sdim}{\sigma\text{-}\dim}
\newcommand{\ld}{\operatorname{ld}}

\newcommand{\ksalg}{k\text{-}\sigma\text{-}\mathsf{Alg}}

\newcommand{\alg}{\mathsf{Alg}}
\newcommand{\ord}{\operatorname{ord}}
\newcommand{\set}{\mathsf{Sets}}
\newcommand{\X}{\mathcal{X}}
\newcommand{\Y}{\mathcal{Y}}

\newcommand{\s}{\sigma}
\newcommand{\de}{\delta}

\newcommand{\groups}{\mathsf{Groups}}
\newcommand{\nn}{\mathbb{N}}
\newcommand{\C}{\mathbb{C}}

\newcommand{\I}{\mathbb{I}}

\newcommand{\hs}{{{}^\sigma\!}}
\newcommand{\VV}{\mathcal{V}}

\newcommand{\ks}{$k$-$\s$}

\newcommand{\Gm}{\mathbb{G}_m}

\newcommand{\iar}{\hookrightarrow}
\newcommand{\Ga}{\mathbb{G}_a}
\def\Sl{\operatorname{SL}}
\def\H{\mathcal{H}}

\def\wtilde{\widetilde}

\begin{document}

\maketitle

\begin{abstract} \let\thefootnote\relax
	We establish several finiteness properties of groups defined by algebraic difference equations.
	One of our main results is that a subgroup of the general linear group defined by possibly infinitely many algebraic difference equations in the matrix entries can indeed be defined by finitely many such equations. As an application, we show that the difference ideal of all difference algebraic relations among the solutions of a linear differential equation is finitely generated.
\end{abstract}




\section*{Introduction}

 Similar to the case of affine algebraic groups, affine difference algebraic groups can be realized as subgroups of general linear groups. However, the defining equations here are not simply polynomials in the matrix entries but difference polynomials, i.e., the defining equations involve a formal symbol $\s$ that has to be interpreted as a ring endomorphism. For example, if $G$ is the difference algebraic subgroup of $\Gl_n$ defined by the algebraic difference equations $X\s(X)^{\operatorname{T}}=\s(X)^{\operatorname{T}}X=I_n$ and $\s\colon\mathbb{C}\to\mathbb{C}$ is the complex conjugation map, then $G(\mathbb{C},\sigma)$ is the group of all complex unitary $n\times n$-matrices. 

Alternatively, affine difference algebraic groups can be described as affine group schemes with a certain additional structure (the difference structure). As schemes, they are typically not of finite type, but they enjoy a certain finiteness property with respect to the difference structure; they are ``of finite $\s$-type''.

A difference field is a field equipped with an endomorphism $\s$, a classical example is $\C(x)$ with $\s(f(x))=f(x+1)$. From an algebraic point of view, an affine difference algebraic group $G$ over a difference field $k$  corresponds to a finitely $\s$-generated \ks-Hopf algebra, i.e., a Hopf algebra $k\{G\}$ over $k$ together with a ring endomorphism $\s\colon k\{G\}\to k\{G\}$ that extends $\s\colon k\to k$. The structure maps of the Hopf algebra are required to commute with $\s$, and $k\{G\}$ is required to be finitely $\s$-generated over $k$, i.e., there exists a finite set $B\subseteq k\{G\}$ such that $B,\s(B),\s^2(B),\ldots$ generates $k\{G\}$ as a $k$-algebra.
A difference ideal $\ida$ of $k\{G\}$ (i.e., an ideal of $k\{G\}$ with $\s(\ida)\subseteq\ida$) is called finitely $\s$-generated if there exists a finite subset $F$ of $\ida$ such that $F,\s(F),\s^2(F),\ldots$ generates $\ida$ as an ideal.
For clarity, we state here our main finiteness results in the language of Hopf algebras.

\begin{theoremA}
	Let $k\{G\}$ be a finitely $\s$-generated \ks-Hopf algebra. Then,
	\begin{enumerate}
		\item every difference ideal of $k\{G\}$ that is a Hopf ideal is finitely $\s$-generated,
		\item every \ks-Hopf subalgebra of $k\{G\}$ is finitely $\s$-generated,
		\item there are only finitely many minimal prime difference ideals in $k\{G\}$.
	\end{enumerate}	
\end{theoremA}

In fact, we prove a very useful stronger version of (i). For a difference algebraic subgroup $G$ of $\Gl_n$, this stronger version takes the following form: For $i\gg 0$ the ideal $\I(G)[i]$ of all difference polynomials in the matrix entries of order at most $i$ that vanish on $G$, is generated by $\I(G)[i-1]$ and $\s(\I(G)[i-1])$.


A difference ideal in a finitely $\s$-generated \ks-algebra need not be finitely $\s$-generated. A well-known counterexample is the difference ideal generated by $y\s(y),y\s^2(y),\ldots$. 
The classical basis theorem in difference algebra (\cite[Theorem 2.5.11]{Levin:difference}) only applies to so-called perfect difference ideals: Every ascending chain of perfect difference ideals in a finitely $\s$-generated \ks-algebra is finite. Recall that a difference ideal $\ida$ is perfect if $f\sigma(f)\in\ida$ implies $f$ in $\ida$. An interesting generalization of the classical basis theorem was suggested by E. Hrushovski in his seminal work \cite{Hrushovski:elementarytheoryoffrobenius}:

\begin{conjB}
	Every ascending chain of radical well-mixed difference ideals in a finitely $\s$-generated \ks-algebra is finite.
\end{conjB}

Recall that a difference ideal $\ida$ is well-mixed if $fg\in \ida$ implies $f\s(g)\in\ida$. In \cite{Levin:OnTheAscendingChainCondition} it is shown that the above conjecture fails if the assumption that the ideals are radical is dropped. On the positive side, E. Hrushovski proved the conjecture under certain addition assumptions (\cite[Lemma 4.35]{Hrushovski:elementarytheoryoffrobenius}). Moreover, the conjecture (appropriately reformulated) has been proved for monomial and binomial difference ideals (\cite{Wang:MonomialDifferenceIdeals},\cite{Wang:FiniteBasisForRadicalWellmixedDifferenceIdeals}).

We show that the Conjecture B is equivalent to: Every finitely $\s$-generated \ks-algebra has only finitely many minimal prime difference ideals. Thus, point (iii) of Theorem A can be seen as providing another positive special case of Conjecture B.

\medskip

Difference algebraic groups are the discrete analog of differential algebraic groups, i.e., groups defined by algebraic differential equations. Differential algebraic groups have always played an important role in differential algebra (see, e.g., \cite{Cassidy:differentialalgebraicgroups}, \cite{Sit:DifferentialAlgebraicSubgroupsofSL2}, \cite{Cassidy:TheDifferentialRationalRepresentationAlgebraOnALinearDifferentialAlgebraicGroup}, \cite{Cassidy:UnipotentDifferentialAlgebraicGroups}, \cite{Kolchin:differentialalgebraicgroups}, \cite{Cassidy:TheClassificationOfTheSemisimpleDifferentialAlgebraicGroups} \cite{Buium:DifferntialAlgebraicGroupsOfFiniteDimension}, \cite{Buium:GeometryOfDifferentialPolynomialFunctionsIAgebraicGroups}) and are also presently an active area (see, e.g., \cite{Pillay:SomeFoundationalQuestionsConcerningDifferentialAlgebraicGroups}, \cite{KowalskiPillay:ProalgebraicAndDifferentialAlgebraicGroupStructuresOnAffineSpaces}, \cite{Ovchinnikov:DifferentialTannakianCategories}, \cite{CassidySinger:AJordanHoelderTheoremForDifferentialAlgebraicGroups}, \cite{MinchenkiOvchinnikov:ZariskiClosuresOfReductiveLinearDifferentialAlgebraiGroups}, \cite{Kamensky:TannakianFormalismOverFieldsWithOperators} \cite{MinchenkoOvchinnikov:ExtensionsOfDifferentialRepresentationsOfSl2andTori}, \cite{Freitag:IndecomposabilityForDifferentialAlgebraicGroups}, \cite{Minchenko:OnCentralExtensionsOfSimpleDifferentialAlgebraicGroups},\cite{Anand:ThePicardVessiotTheoryConstrainedCohomologyandLinearDifferentialAlgebriacGroups},\cite{MinchenkoOvchinnikov:TrivialityOfDifferentialGaloisCohomologyOfLinearDifferentialAlgebraicGroups},\cite{ChatzidakisPillay:GeneralizedPicardVessiotExtensionsAndDifferentialGaloisCohomology}.) See also \cite{Malgrange:DifferentialAlgebraicGroups} for a more geometric approach to differential algebraic groups and \cite{Buium:DifferentialSubgroupsOfSimpleAlgebraicGroupsOverPadicFields} for an arithmetic analog of difference/differential algebraic groups. Moreover, Galois theories for differential or difference equations where the Galois groups are differential algebraic groups (\cite{Pillay:DifferentialGaloisTheory1}, \cite{Landesman:GeneralizedDifferentialGaloisTheory}, \cite{CassidySinger:GaloisTheoryofParameterizedDifferentialEquations}, \cite{HardouinSinger:DifferentialGaloisTheoryofLinearDifferenceEquations}) have recently given a new impetus to the study of differential algebraic groups (see, e.g. \cite{MinchenkoOvchinnikovSinger:ReductiveLinearDifferentialAlgebraicGroupsAndTheGaloisGroups}, \cite{MinchenkoOvchinnikovSinger:UnipotentDifferentialAlgebraicGroupsAsParameterizedDifferentialGaloisGroups}, \cite{GorchinskyOvchinnikov:IsomonodromicDifferentialEquationsAndDifferentialCategories},\cite{HardouinMinchenkoOvchinnikov:CalculatingDifferentialGaloisGroupsOfParameterizedDifferentialEqautionsWithApplications}).

In contrast to the situation in differential algebra, difference algebraic groups have long played no role at all in difference algebra. The author can only speculate why. One reason might be that the traditional definition of a difference variety (see \cite[Chapter 4]{Cohn:difference}) is not practical for studying groups. 

Around the turn of the century a considerable interest in the model theory of difference fields emerged. (See e.g.,  \cite{Macintyre:GenricAutomorphismsOfFields}, \cite{ChatzidakisHrushovski:ModelTheoryOfDifferenceFields}, \cite{ChatzidakiHrushovskiPeterzil:ModelTheoryofDifferenceFieldsIIPeriodicIdelas}.) Groups definable in ACFA, the model companion of the theory of difference fields, played a crucial role in remarkable applications of model theory to number theory, especially regarding the Manin-Mumford conjecture. See \cite{Hrushovski:TheManinMumfordConjectureAndTheModelTheorOfDifferenceFields},
\cite{Bouscaren:TheorieModesEtManinMumford}, \cite{Chatzidakis:DifferencefieldsModelTheoryAndApplicationsToNumberTheory},
\cite{Chatzidakis:GroupsDefinableInACFA}, \cite{Scanlon:APositiveCharacteristicManinMumfordTheorem}, \cite{Scanlon:LocalAndreOrtConjectureForTheUniversalAbelianVariety}, \cite{Scanlon:DifferenceAlgebraicSubgroupsOfCommutativeAlgebraicGroups}, \cite{KowalskiPillay:OnAlgebraicSigmaGroups}.
Groups definable in ACFA have further been studied in \cite{KowalskiPillay:ANoteonGroupsDefinableInDifferenceFields} and
\cite{ChatzidakisHrushovski:OnSubgroupsOfSemiabelianVarietiesDefinedByDifferenceEquations}.
%

 Categories of representations of affine difference algebraic groups have been studied in \cite{Kamensky:TannakianFormalismOverFieldsWithOperators} and \cite{OvchinnikovWibmer:TannakianCategoriesWithSemigroupActions}.  A cohomology theory for such representations and more general actions has been developed in \cite{ChalupnikKowalski:DifferenceModulesAndDifferenceCohomology}. Furthermore, torsors for affine difference algebraic groups are studied in \cite{BachmayrWibmer:TorsorsForDifferenceAlgebraicGroups}. 

\medskip

Understanding the relations among the solutions of a given linear differential equation is a classical and important problem. Often the interesting relations among the solutions are not simply algebraic relations but involve transformations of the variable or the parameters of the differential equation.
For example, consider Bessel's differential equation

$$
x^2y''+xy'+(x^2-\alpha^2)y=0.
$$
The solution $J_\alpha(x)$, the Bessel function of the first kind, satisfies the well-known linear recurrence relation

$$
xJ_{\alpha+2}(x)-2(\alpha+1)J_{\alpha+1}(x)+xJ_\alpha(x)=0.
$$
This is a \emph{difference algebraic relation} with respect to the transformation \mbox{$\alpha\mapsto\alpha+1$} on the parameter $\alpha$. Many classical functional identities can be interpreted as difference algebraic relations among solutions of linear differential equations.  For example, the relation 
$$
\cos(2x)=2\cos^2(x)-1
$$
is a \emph{difference algebraic relation} with respect to the transformation \mbox{$x\mapsto 2x$} on the variable $x$, satisfied by the solution $\cos(x)$ of $y''+y=0$.

A Galois theory for linear differential equations that is able to handle this kind of difference algebraic relations has been introduced in \cite{DiVizioHardouinWibmer:DifferenceGaloisTheoryOfLinearDifferentialEquations}. There is also a similar Galois theory for linear difference equations (\cite{OvchinnikovWibmer:SGaloisTheoryOfLinearDifferenceEquations}). The Galois groups in these Galois theories are affine difference algebraic groups and they measure the difference algebraic relations among the solutions. 

Based on this Galois theory, we will show, as an application of point (i) of Theorem A, that, for a given linear differential equation, there exists a \emph{finite} set $F$ of difference algebraic relations among the solutions such that any other difference algebraic relation is a consequence of the difference algebraic relations in $F$. Or, formulated more algebraically, we show that the difference ideal of all difference algebraic relations among the solutions of a linear differential equation is finitely generated.

As illustrated in \cite{DiVizioHardouinWibmer:DifferenceAlgebraicRel} and \cite{OvchinnikovWibmer:SGaloisTheoryOfLinearDifferenceEquations}, the Galois theories described above make it possible to use structure results about affine difference algebraic groups to analyze and classify the possible difference algebraic relations among the solutions of certain linear differential and difference equations. In this respect, the understanding of the Zariski dense difference algebraic subgroups of a given affine algebraic group is highly relevant. For example, some understanding of the Zariski dense difference closed subgroups of $\operatorname{SL}_2$, originating from \cite{ChatzidakiHrushovskiPeterzil:ModelTheoryofDifferenceFieldsIIPeriodicIdelas}, is a key ingredient in proving that any two linearly independent solutions of the Airy equation are difference algebraically independent (\cite[Corollary 6.10]{DiVizioHardouinWibmer:DifferenceAlgebraicRel}). 


Despite the recent works on difference algebraic groups and their applications to other areas outlined above, the theory of difference algebraic groups is still in its infancy and far from a level comparable to the state of the art of the theory of differential algebraic groups. The purpose of this article is therefore also to lay the groundwork for a further comprehensive study of affine difference algebraic groups\footnote{Some steps in this direction are already contained in the authors habilitation thesis (\cite{Wibmer:Habil}), which encompasses the first six sections of this article.}. In particular, we introduce three basic numerical invariants for affine difference algebraic groups: the difference dimension, the order and the limit degree. These are in fact the three standard numerical invariants in difference algebra. However, in the standard textbooks \cite{Cohn:difference} and \cite{Levin:difference} these are defined in terms of extensions of difference fields and it is not obvious how to generalize the definitions to difference Hopf algebras. Moreover, we establish a dimension theorem for affine difference algebraic groups: If $H_1$ and $H_2$ are difference algebraic subgroups of an affine difference algebraic group $G$, then $\sdim(H_1\cap H_2)\geq \sdim(H_1)+\sdim(H_2)-\sdim(G)$. Interestingly, a similar statement fails for arbitrary difference varieties (\cite[Chapter 8, Section 8]{Cohn:difference}). In regards to further developments in the theory of affine difference algebraic groups, we note that (ii) of Theorem~A is crucial for establishing the existence of quotients (by normal subgroups) in the category of affine difference algebraic groups.

\medskip

We conclude this introduction with an outline of the article: In Section \ref{Section:Preliminaries} we recall the necessary definitions from difference algebra and difference algebraic geometry. Then, in Section~\ref{sec: Difference algebraic groups: Examples and linearity } we introduce affine difference algebraic groups, present several examples and show that every affine difference algebraic group is isomorphic to a difference algebraic subgroup of some general linear group. Section \ref{sec: Zariski closures of difference algebraic groups and the growth group} is the technical heart of this article. We show that the growth of the Zariski closures of a difference algebraic subgroup $G$ of an affine algebraic group is governed by an algebraic group that we term the \emph{growth group}. The difference dimension of $G$ equals the dimension of the growth group and if the latter is zero, the size of the growth group equals the limit degree of $G$. 
 In Section \ref{sec: two finiteness theorems} we prove (i) (including the strong version) and (ii) of Theorem~A.
As an application of the strong version of (i) we prove the dimension theorem for affine difference algebraic groups.
In Section \ref{sec: limit degree} we introduce the limit degree of an affine difference algebraic group and relate our work to \cite{KowalskiPillay:OnAlgebraicSigmaGroups} by showing that the category of affine difference algebraic groups of limit degree one is equivalent to the category of affine algebraic $\s$-groups studied in \cite{KowalskiPillay:OnAlgebraicSigmaGroups}.
Then in Section \ref{section: components}, we discuss a conjecture equivalent to Conjecture B and we prove (iii) of Theorem A. These questions are naturally related to the study of the connected components of an affine difference algebraic group. Finally, in the last section, we present the application of (i) of Theorem A to linear differential equations.

\section{Preliminaries} \label{Section:Preliminaries}

In this section we recall some basic definitions from difference algebra and we define fundamental geometric concepts like difference varieties.

All rings are assumed to be commutative and unital. $\nn=\{0,1,2,\ldots\}$ and $\s^0$ is understood to be the identity map. By an ``algebraic group'' we mean an affine group scheme of finite type over a field. In particular, in positive characteristic, an algebraic group need not be reduced. A morphism $\G\to\H$ of algebraic groups is a \emph{quotient map} if it is faithfully flat (\cite[Def.~ 5.5]{Milne:AlgebraicGroupsTheTheoryOfGroupSchemesOfFiniteTypeOverAField}). Equivalently, the dual map $k[\H]\to k[\G]$ is injective. Quotient maps are the appropriate scheme theoretic analog of surjective morphisms of smooth algebraic groups. (See \cite[Chapter~5]{Milne:AlgebraicGroupsTheTheoryOfGroupSchemesOfFiniteTypeOverAField} for more background.)

\subsection{Difference algebra}

Standard references for difference algebra are \cite{Levin:difference} and \cite{Cohn:difference}.
A \emph{difference ring} (or \emph{$\s$-ring} for short) is a ring $R$  together with a ring endomorphism $\s\colon R\to R$. If $\s$ is an automorphism, then $R$ is called \emph{inversive}. It is customary to use the same symbol $\s$ for diverse endomorphisms.
A \emph{morphism of $\s$-rings} $R$ and $S$ is a morphism $\psi\colon R\to S$ of rings such that
\[
\xymatrix{
	R \ar_{\s}[d] \ar^-{\psi}[r] & S  \ar^{\s}[d] \\
	R \ar^-{\psi}[r] & S \\
}
\]
commutes.
%
If $R$ and $S$ are $\s$-rings such that $R$ is a subring of $S$ and the inclusion map is a morphism of $\s$-rings, then $R$ is a \emph{$\s$-subring} of $S$. 
A \emph{difference field}  (or \emph{$\s$-field} for short) is a difference ring whose underlying ring is a field.
%

Let $k$ be a difference ring. A morphism of difference rings $k\to R$ is also called a \emph{\mbox{\ks-algebra}}. \index{\ks-algebra}A \emph{morphism of \ks-algebras} is a morphism of $k$-algebras that is also a morphism of $\s$-rings.
The category of \ks-algebras is denoted by
$\ksalg$ and $\Hom(R,S)$ denotes the set of morphisms of \ks-algebras from $R$ to $S$.

For \ks-algebras $R$ and $S$ the $k$-algebra $R\otimes_k S$ is naturally a $\s$-ring by $\s(r\otimes s)=\s(r)\otimes\s(s)$ for $r\in R$ and $s\in S$. Indeed, $R\otimes_k S$ is the coproduct of $R$ and $S$ in the category of \ks-algebras.

A \emph{\ks-subalgebra} of a \ks-algebra $R$ is a $k$-subalgebra of $R$ that is also $\s$-subring of $R$. For a subset $F\subseteq R$, the smallest \ks-subalgebra of $R$ that contains $F$ is denoted by $k\{F\}.$
It is called the \emph{\ks-subalgebra of $R$ $\s$-generated by $F$} (over $k$). As a $k$-algebra, $k\{F\}$ is generated by all elements of the form $\s^i(f)$ with $i\in\nn$ and $f\in F$. A \ks-algebra $R$ is called \emph{finitely $\s$-generated} (over $k$) if there exists a finite subset $F\subseteq R$ such that $R=k\{F\}$.

The \emph{ring of $\s$-polynomials over a $\s$-ring $k$ in the $\s$-variables $y_1,\ldots,y_n$} is the polynomial ring
$$k\{y\}=k\{y_1,\ldots,y_n\}=k[y_1,\ldots,y_n,\s(y_1),\ldots,\s(y_n),\s^2(y_1),\ldots,\s^2(y_n),\ldots]$$
in the variables $y_1,\ldots,y_n,\s(y_1),\ldots,\s(y_n),\ldots$ over $k$. It is naturally a \ks-algebra with $\s\colon k\{y\}\to k\{y\}$ extended from $\s\colon k\to k$ as suggested by the naming of the variables.
If $R$ is a \ks-algebra, $f\in k\{y\}$ a $\s$-polynomial and $x=(x_1,\ldots,x_n)\in R^n$, the element $f(x)\in R$ is obtained by substituting $\s^i(x_j)$ for $\s^i(y_j)$. If $f(x)=0$, then $x$ is a \emph{solution} of $f$.

%

An ideal $\ida$ of a $\s$-ring $R$ is called a \emph{difference ideal} \index{difference ideal} (or \emph{$\s$-ideal} \index{$\s$-ideal} for short) if $\s(\ida)\subseteq \ida$. In this case $R/\ida$ naturally inherits the structure of a $\s$-ring such that the canonical map $R\to R/\ida$ is a morphism of $\s$-rings.
%
%
Let $F$ be a subset of $\s$-ring $R$. The smallest $\s$-ideal of $R$ containing $F$ is denoted by
$[F]$
and called the \emph{$\s$-ideal $\s$-generated by $F$} \index{$\s$-ideal $\s$-generated by $F$}.
As an ideal $[F]$ is generated by all elements of $R$ that are of the form $\s^i(f)$ with $i\in\nn$ and $f\in F$.
A $\s$-ideal $\ida$ of a $\s$-ring $R$ is \emph{finitely $\s$-generated} if $\ida=[F]$ for a finite subset $F\subseteq\ida$.

%
%

\subsection{Difference algebraic geometry} 

The various frameworks for algebraic geometry (e.g., Weil-style algebraic geometry, schemes, etc.) all have difference analogs. While the standard references \cite{Levin:difference} and \cite{Cohn:difference}
follow the style of Weil, in recent years, most authors follow a more utilitarian perspective and choose the geometric framework that is best suited for their problem at hand.
It is well recognized that a functorial-schematic approach to algebraic groups has its benefits (\cite{Grothendieck:SGA3_1},  \cite{DemazureGabriel:GroupesAlgebriques}, \cite{Waterhouse:IntroductiontoAffineGroupSchemes}, \cite{Jantzen:RepresentationsOfAlgebraicGroups}, \cite{Milne:AlgebraicGroupsTheTheoryOfGroupSchemesOfFiniteTypeOverAField}) and this is the strategy we will follow here. In other words, our approach to difference algebraic geometry is a natural adaptation of the approach of the above references from algebra to difference algebra, i.e., from algebraic equations to difference algebraic equations. Our approach is identical or equivalent to the approaches in \cite{BachmayrWibmer:TorsorsForDifferenceAlgebraicGroups}, \cite{Kamensky:TannakianFormalismOverFieldsWithOperators}, \cite{OvchinnikovWibmer:TannakianCategoriesWithSemigroupActions}, \cite{MoosaScanlon:GeneralizedHasseSchmidtVarietiesAndTheirJetSpaces}, \cite{DiVizioHardouinWibmer:DifferenceGaloisTheoryOfLinearDifferentialEquations} and \cite{ChalupnikKowalski:DifferenceModulesAndDifferenceCohomology}. However, it differs (in a more than formal way) from the approaches in \cite{Hrushovski:elementarytheoryoffrobenius}, \cite{Levin:difference} and \cite{ChatzidakisHrushovski:ModelTheoryOfDifferenceFields}.
%
%
%
%
%

Throughout the article we work over a fixed $\s$-ring $k$. We are mainly interested in the case that $k$ is a difference field. 

Essentially, a \emph{difference variety} (over $k$) is the set of solutions of a system of algebraic difference equations (over $k$). Let us make this more precise:
Let $R$ be a $k$-$\s$-algebra and $F\subseteq k\{y\}=k\{y_1,\ldots,y_n\}$ a set of $\s$-polynomials over $k$. Then we may consider the $R$-rational solutions of $F$, that is
$$\V_R(F)=\{ x\in R^n|\ f(x)=0 \text{ for all } f\in F\}.$$
Note that $R\rightsquigarrow\V_R(F)$ is naturally a functor from $k$-$\s$-$\alg$ to $\set$, the category of sets. We denote this functor by $\V(F)$.

\begin{defi} \label{defi: svariety}
	Let $k$ be a $\s$-ring. A \emph{difference variety} (or \emph{$\s$-variety}\index{$\s$-variety} for short) over $k$ is a functor from $k$-$\s$-$\alg$ to $\set$, that is of the form $\V(F)$, for some $n\geq 1$ and $F\subseteq k\{y_1,\ldots,y_n\}$. A morphism of $\s$-varieties is a morphism of functors.
\end{defi}
It would be more accurate to add the word ``affine'' into Definition \ref{defi: svariety}. However, to avoid endless iterations of the word ``affine'' we choose not to do so.
%

Let $F,G\subseteq k\{y\}$ and $X=\V(F)$, $Y=\V(G)$. If $F\subseteq G$, then $Y(R)\subseteq X(R)$ for every $k$-$\s$-algebra $R$ and $Y$ is a subfunctor of $X$. In this situation, $Y$ is called  a \emph{$\s$-closed $\s$-subvariety}\index{$\s$-closed $\s$-subvariety} of $X$ and write $Y\subseteq X$.


Let $X=\V(F)$ be a $\s$-variety, where $F\subseteq k\{y\}=k\{y_1,\ldots,y_n\}$. Then
\begin{equation} \label{eqn: I(X)}
\I(X)=\{f\in k\{y\}| \ f(x)=0 \ \text{ for all $k$-$\s$-algebras $R$ and all } x\in X(R) \}
\end{equation}
is a $\s$-ideal of $k\{y\}$. The $k$-$\s$-algebra
$$k\{X\}=k\{y\}/\I(X)$$
is called the \emph{coordinate ring}\index{coordinate ring} of $X$. As we may choose $R=k\{y\}/[F]$ in (\ref{eqn: I(X)}), we see that $\I(X)=[F]\subseteq k\{y\}$.

Let $R$ be a $k$-$\s$-algebra and let $\overline{y}=(\overline{y_1},\ldots,\overline{y_n})\in k\{X\}^n$ denote the images of $y_1,\ldots,y_n$ in $k\{X\}$. The bijection
$$\Hom(k\{X\},R)\to X(R)$$
thats maps a morphism $\psi\colon k\{X\}\to R$ of $k$-$\s$-algebras to $\psi(\overline{y})$ is functorial in $R$. Thus the functor $X$ is represented by $k\{X\}$. Conversely, since every finitely $\s$-generated \mbox{$k$-$\s$-algebra} can be written in the form $k\{y\}/[F]$, we see that a functor from $k$-$\s$-$\alg$ to $\set$ that is representable by a finitely $\s$-generated $k$-$\s$-algebra is isomorphic (as a functor) to a \mbox{$\s$-variety}.

In the sequel we will allow ourselves the little abuse of notation to also call a functor isomorphic to a $\s$-variety a $\s$-variety.
In particular, we will often identify $X$ with the functor $\Hom(k\{X\},-)$.
Thus a functor $X$ from $k$-$\s$-$\alg$ to $\set$ is a $\s$-variety if and only if it is representable by a finitely $\s$-generated $k$-$\s$-algebra $k\{X\}$. It is then clear from the Yoneda lemma that:

\begin{rem} \label{rem: equivalence of categories for svarieties}
	The category of $\s$-varieties over $k$ is anti-equivalent to the category of finitely $\s$-generated $k$-$\s$-algebras.
\end{rem}

If $\f\colon X\to Y$ is a morphism of $\s$-varieties over $k$, the dual morphism of $k$-$\s$-algebras is denoted by
$\f^*\colon k\{Y\}\to k\{X\}.$
As the tensor product is the coproduct in the category of \ks-algebras it follows that:

\begin{rem} \label{rem: products}
	The category of $\s$-varieties has products. Indeed, if $X$ and $Y$ are \mbox{$\s$-varieties} over $k$, then $k\{X\times Y\}=k\{X\}\otimes_k k\{Y\}$. Moreover, there is a terminal object, namely, the functor represented by the $k$-$\s$-algebra $k$.
\end{rem}

In other words, the functor $R\rightsquigarrow(X\times Y)(R)=X(R)\times Y(R)$ is represented by $k\{X\}\otimes_k k\{Y\}$.

%


Let $Y$ be a $\s$-variety and $f\in k\{Y\}$. Then, for any $k$-$\s$-algebra $R$, we have a well-defined map $f\colon Y(R)\to R$ given by evaluating a representative of $f$ in $k\{y_1,\ldots,y_n\}$ at an element of $Y(R)\subseteq R^n$. In a coordinate free manner $f\colon Y(R)\to R$ can be described as the map that sends $\psi\in Y(R)=\Hom(k\{Y\},R)$ to $\psi(f)\in R$.

Let $X$ be a $\s$-closed $\s$-subvariety of $Y$. Then
\begin{equation} \label{eqn: I(X)2}
\I(X)=\{f\in k\{Y\}| \ f(x)=0 \ \text{ for all $k$-$\s$-algebras $R$ and all } x\in X(R) \}
\end{equation}
is a $\s$-ideal of $k\{Y\}$. We call $\I(X)\subseteq k\{Y\}$ the \emph{defining ideal of $X$}\index{defining ideal of a $\s$-variety} (in $k\{Y\}$).
This notation is consistent with (\ref{eqn: I(X)}) in the sense that $\I(X)$ as defined in (\ref{eqn: I(X)}) is the defining ideal of $X$ in $k\{y\}=k\{y_1,\ldots,y_n\}$. Moreover, $\I(X)\subseteq k\{Y\}$ agrees with the image in $k\{Y\}=k\{y\}/\I(Y)$ of the defining ideal of $X$ in $k\{y\}$. So $k\{X\}=k\{Y\}/\I(X)$.

Conversely, let $F\subseteq k\{Y\}$. Then we can define a $\s$-closed $\s$-subvariety $\V(F)$ of $Y$ by
$$\V(F)(R)=\{x\in Y(R)|\ f(x)=0 \text{ for all } f\in F\}$$
for any $k$-$\s$-algebra $R$.

\begin{lemma} \label{lemma: correspondence sideal ssubvarriety}
	Let $Y$ be a $\s$-variety. Then $\I$ and $\V$ are mutually inverse bijections between the set of $\s$-closed $\s$-subvarieties of $Y$ and the set of $\s$-ideals of $k\{Y\}$.
\end{lemma}
\begin{proof}
	Let $\ida$ be a $\s$-ideal of $k\{Y\}$. Clearly $\ida\subseteq\I(\V(\ida))$. Since we may choose $R=k\{Y\}/\ida$ in (\ref{eqn: I(X)2}) it follows that $\ida=\I(\V(\ida))$.
	
	Let $X$ be $\s$-closed $\s$-subvariety of $Y$. Then $X=\V(\ida)$ for some $\s$-ideal $\ida$ of $k\{Y\}$. So $\V(\I(X))=\V(\I(\V(\ida)))=\V(\ida)=X$.
\end{proof}

Note that if $X$ is a $\s$-closed $\s$-subvariety of $Y$ and $R$ a $k$-$\s$-algebra, then $X(R)\subseteq Y(R)$ corresponds to $\{\psi\in\Hom(k\{Y\},R)|\ \I(X)\subseteq\ker(\psi)\}\subseteq\Hom(k\{Y\},R)$.
%
If $Y$ and $Z$ are $\s$-closed $\s$-subvarieties of a $\s$-variety $X$, then we can define a subfunctor
$Y\cap Z$
of $X$ by $R\rightsquigarrow Y(R)\cap Z(R)$. Then $Y\cap Z$ is a $\s$-closed $\s$-subvariety of $X$, indeed, $\I(Y\cap Z)\subseteq k\{X\}$ is the ideal generated by $\I(Y)$ and $\I(Z)$.
%
%
%
%
%
%

The following lemma defines the difference analog of the scheme theoretic image (\cite[Tag 01R5]{stacks-project}).
\begin{lemma} \label{lemma: f(X)}
	Let $\f\colon X\to Y$ be a morphism of $\s$-varieties. There exists a unique $\s$-closed $\s$-subvariety $\f(X)$ of $Y$ with the following property. The morphism $\f$ factors through $\f(X)$ and if $Z\subseteq Y$ is a $\s$-closed $\s$-subvariety such that $\f$ factors through $Z$ then $\f(X)\subseteq Z$.
\end{lemma}
\begin{proof}
	Let $\ida$ denote the kernel of $\f^*\colon k\{Y\}\to k\{X\}$ and set $\f(X)=\V(\ida)\subseteq Y$. As $\f^*$ factors through $k\{Y\}\to k\{\f(X)\}=k\{Y\}/\ida$, we see that $\f$ factors through $\f(X)$. If $\f$ factors through $Z$, i.e., $\f^*$ factors through $k\{Y\}\to k\{Z\}=k\{Y\}/\I(Z)$, then clearly $\I(Z)\subseteq\ida$. So $\f(X)\subseteq Z$.
\end{proof}

%

A morphism $\f\colon X\to Y$ of $\s$-varieties is called a \emph{$\s$-closed embedding}\index{$\s$-closed embedding} if $\f$ induces an isomorphism between $X$ and a $\s$-closed $\s$-subvariety of $Y$, i.e., $X\to \f(X)$ is an isomorphism.
We write $\f\colon X\hookrightarrow Y$ to express that $\f$ is a $\s$-closed embedding. In analogy to a well known result in algebraic geometry we have:

\begin{lemma} \label{lemma: sclosed embedding}
	A morphism $\f\colon X\to Y$ of $\s$-varieties is a $\s$-closed embedding if and only if $\f^*\colon k\{Y\}\to k\{X\}$ is surjective.
\end{lemma}
\begin{proof}
	Let $\ida$ denote the kernel of $\f^*\colon k\{Y\}\to k\{X\}$. The dual map to $X\to\f(X)$ is $k\{Y\}/\ida\to k\{X\}$. It is an isomorphism if and only if $\f^*$ is surjective.
\end{proof}

Base change works as expected: Let $k'$ be a \ks-algebra and $X$ a $\s$-variety over $k$. 
We can define a functor
$X_{k'}$
from $k'$-$\s$-$\alg$ to $\set$ by $X_{k'}(R')=X(R')$ for any $k'$-$\s$-algebra $R'$. Then $X_{k'}$ is a $\s$-variety over $k'$. Indeed,
$k'\{X_{k'}\}=k\{X\}\otimes_k k'.$
In terms of equations, this means that a system $F\subseteq k\{y\}=k\{y_1,\ldots,y_n\}$ of algebraic difference equations over $k$ is considered as a system $F\subseteq k'\{y\}$ of algebraic difference equations over $k'$. So if $k\{X\}=k\{y\}/[F]$ then $k'\{X_{k'}\}=k'\{y\}/[F]$.

\subsection{Zariski closures} \label{subsec: Zariski closures}

\emph{From now on we assume that our base $\s$-ring $k$ is a $\s$-field.}

\medskip

Since algebraic equations can be interpreted as algebraic difference equations, it is clear that an affine scheme of finite type over $k$, can be interpreted as a $\s$-variety over $k$. In this subsection we make this more precise and we introduce some notation that will be useful in the coming sections. In particular, we define the Zariski closures of a $\s$-subvariety of a variety (cf. Section 4.3 in \cite{Hrushovski:elementarytheoryoffrobenius} and Sections A.4 and A.5 in \cite{DiVizioHardouinWibmer:DifferenceGaloisTheoryOfLinearDifferentialEquations}).

For a $k$-$\s$-algebra $R$, let
$R^\sharp$ denote the $k$-algebra obtained from $R$ by forgetting the ring endomorphism $\s\colon R\to R$.
Let $\X$ be an affine scheme of finite type over $k$.
We can define a functor $[\s]_k\X$ from $k$-$\s$-$\alg$ to $\set$ by
$$[\s]_k\X(R)=\X(R^\sharp)$$
for any $k$-$\s$-algebra $R$. We will show that $[\s]_k\X$ is a $\s$-variety over $k$. That is, we will construct $k\{[\s]_k\X\}$. 
Let $A$ be a $k$-algebra. For every $i\geq 0$ let
$${}^{\s^i}\! A=A\otimes_k k$$ where the tensor product is formed by using $\s^i\colon k\to k$ on the right hand side. We consider ${}^{\s^i\!}A$ as $k$-algebra via the right factor. We set
$$A[i]=A\otimes_k{}^{\s\!}A\otimes_k \cdots\otimes_k{}^{\s^i\!}A.$$
We have inclusions $A[i]\hookrightarrow A[i+1]$ of $k$-algebras and the limit, i.e, the union
$[\s]_kA $
of the $A[i]$'s is a $k$-$\s$-algebra, where for $(r_0\otimes\lambda_0)\otimes\cdots\otimes(r_i\otimes\lambda_i)\in {}^{\s^0\!}A\otimes_k \cdots\otimes_k{}^{\s^i\!}A=A[i]$ the map $\s\colon [\s]_kA\to [\s]_kA$ is given by
$$\s((r_0\otimes\lambda_0)\otimes\cdots\otimes (r_i\otimes\lambda_i))=(1\otimes 1)\otimes(r_0\otimes\s(\lambda_0))\otimes \cdots \otimes(r_i\otimes\s(\lambda_i))\in A[i+1].$$
The inclusion $A=A[0]\hookrightarrow [\s]_kA$ is characterized by the following universal property:
\begin{lemma} \label{lemma: universal property of sA}
	For every $k$-algebra $A$ there exists a $k$-$\s$-algebra $[\s]_kA$ and a morphism $\psi\colon A\to [\s]_kA$ of $k$-algebras such that for every $k$-$\s$-algebra $R$ and every morphism $\psi'\colon A\to R$ of $k$-algebras there exists a unique morphism $\varphi\colon [\s]_kA\to R$ of $k$-$\s$-algebras making
	\[
	\xymatrix{
		A \ar[rr]^-\psi \ar[rd]_{\psi'} & & [\s]_kA \ar@{..>}[ld]^\varphi \\
		& R &
	}
	\]
	commutative. \qed
\end{lemma}
In other words,
\begin{equation} \label{eqn: adjoin} \Hom([\s]_kA,R)\simeq\Hom(A,R^\sharp)\end{equation}
and $[\s]_k$ is left adjoint to the forgetful functor $(-)^\sharp$.

\begin{ex}
	If $A=k[y_1,\ldots,y_n]$, then $[\s]_k A=k\{y_1,\ldots,y_n\}$. More generally, if $A=k[y_1,\ldots,y_n]/(F)$, then $[\s]_k A=k\{y_1,\ldots,y_n\}/[F]$.
\end{ex}

Let $k[\X]$ denote the coordinate ring of $\X$, i.e., $\X=\spec(k[\X])$. Then (\ref{eqn: adjoin}) with $A=k[\X]$ shows that $[\s]_k\X$ is represented by $[\s]_k k[\X]$. If $F\subseteq k[\X]$ generates $k[\X]$ as a $k$-algebra, then $F\subseteq [\s]_k k[\X]$ generates $[\s]_k k[\X]$ as a $k$-$\s$-algebra. Therefore $[\s]_k\X$ is a $\s$-variety over $k$. Indeed, $k\{[\s]_k\X\}=[\s]_k k[\X]$. In the sequel, if confusion is unlikely, we will often write $\X$ instead of $[\s]_k\X$. In particular, we will write $k\{\X\}$ instead of $k\{[\s]_k\X\}$ and by a $\s$-closed $\s$-subvariety of $\X$, we mean a $\s$-closed $\s$-subvariety of $[\s]_k\X$.

%

To define the Zariski closures of a $\s$-closed $\s$-subvariety of $\X$ we will use notations for schemes similarly to the ones introduced for $k$-algebras above: For $i\geq 0$, we denote with ${}^{\s^i\!}\X $
the scheme over $k$ obtained from $\X$ by base change via $\s^i\colon k\to k$. Similarly, if $\f\colon \X\to \Y$ is a morphism of schemes over $k$, then ${}^{\s^i\!}\f\colon{}^{\s^i\!}X\to{}^{\s^i\!}\Y$ denotes the morphism of schemes over $k$ obtained from $\f$ by base change via $\s^i\colon k\to k$. We also set
$\X[i]=\X\times\hs\X\times\cdots\times{}^{\s^i\!}\X.$

Let $Y$ be a $\s$-closed $\s$-subvariety of $\X$. Then $Y$ is defined by a $\s$-ideal $\I(Y)\subseteq k\{\X\}=\cup_{i\geq 0}k[\X[i]]$. For $i\geq 0$, the closed subscheme
$Y[i]$ of $\X[i]$ defined by $\I(Y)\cap k[\X[i]]\subseteq k[\X[i]]$ is called the \emph{$i$-th order Zariski closure of $Y$ in $\X$}\index{$i$-th order Zariski closure}.

For $i\geq 1$ we have morphisms of schemes over $k$
$$\pi_i\colon\X[i]\to\X[i-1],\ (x_0,\ldots,x_i)\mapsto (x_0,\ldots,x_{i-1})$$ and
$$\s_i\colon\X[i]\to{\hs(\X[i-1])}={{}^{\sigma}\!\X}\times\cdots\times{{}^{\sigma^i}\!\X},\ (x_0,\ldots,x_i)\mapsto (x_1,\ldots,x_i)$$ that form a commutative diagram:
\[
\xymatrix{
	\X[i] \ar_{\s_i}[d] \ar^-{\pi_i}[r] & \X[i-1]  \ar^{\s_{i-1}}[d] \\
	{}^{\s\!}(\X[i-1]) \ar^-{{}^{\s\!}\pi_{i-1}}[r] & {}^{\s\!}(\X[i-2]) \\
}
\]
We have induced morphisms $\pi_i\colon Y[i]\to Y[i-1]$, and since $\I(Y)\subseteq k\{\X\}$ is a $\s$-ideal, we also have induced morphisms $\s_i\colon Y[i]\to {}^{\s\!}(Y[i-1])$.

\section{Difference algebraic groups: Examples and linearity} \label{sec: Difference algebraic groups: Examples and linearity }

%
%
%
%

In this section we define $\s$-algebraic groups and present several examples. We also show that every $\s$-algebraic group is isomorphic to a $\s$-closed subgroup of some general linear group.


We work over a fixed $\s$-field $k$. Since the category of $\s$-varieties has products and a terminal object (Remark \ref{rem: products}) we can make the following definition.

\begin{defi}
	A \emph{$\s$-algebraic group} over $k$ is a group object in the category of $\s$-varieties over $k$. 
\end{defi}
Alternatively, we could define a $\s$-algebraic group as a functor from $k$-$\s$-$\alg$ to $\groups$, the category of groups, that is representable by a finitely $\s$-generated $k$-$\s$-algebra.
A \emph{morphism of $\s$-algebraic groups} is a morphism of $\s$-varieties that respects the group structure. A \emph{\mbox{$\s$-closed} embedding} $\f\colon G\iar H$ of \mbox{$\s$-al}gebraic groups is a morphism of \mbox{$\s$-al}gebraic groups that is a $\s$-closed embedding of $\s$-varieties.

Let $G$ be a $\s$-algebraic group. A \emph{$\s$-closed subgroup $H$ of $G$} \index{$\s$-closed subgroup} is a $\s$-closed $\s$-subvariety $H$ of $G$ such that $H(R)$ is a subgroup of $G(R)$ for any $k$-$\s$-algebra $R$. Then $H$ itself is a $\s$-algebraic group. We write $H\leq G$
to express that $H$ is a $\s$-closed subgroup of $G$. 

The duality between affine group schemes and Hopf algebras carries over to the difference world in a straight forward manner.

\begin{defi} \label{defi: ksHopfalgebra}
	A \emph{$k$-$\s$-Hopf algebra} is a Hopf algebra $A$ over $k$ with the structure of a $k$-$\s$-algebra such that the Hopf-algebra structure maps $\Delta\colon A\to A\otimes_k A$, $S\colon A\to A$ and $\varepsilon\colon A\to k$ are morphisms of $k$-$\s$-algebras.
\end{defi}
A \emph{$k$-$\s$-Hopf subalgebra} of a $k$-$\s$-Hopf algebra is a $k$-$\s$-subalgebra that is Hopf subalgebra. A \emph{$\s$-Hopf ideal} of a \ks-Hopf algebra is a Hopf ideal that is also $\s$-ideal.

\begin{rem} \label{rem: equivalence sgroups sHopfalgebras}
	The category of $\s$-algebraic groups over $k$ is anti-equivalent to the category of $k$-$\s$-Hopf algebras that are finitely $\s$-generated over $k$.
\end{rem}
\begin{proof}
	This follows from Remark \ref{rem: equivalence of categories for svarieties} in an analogous fashion to Theorem 1.4 in \cite{Waterhouse:IntroductiontoAffineGroupSchemes}.
\end{proof}
%
 In analogy to Section 2.1 in \cite{Waterhouse:IntroductiontoAffineGroupSchemes}, we obtain the following lemma from Lemma~\ref{lemma: correspondence sideal ssubvarriety}.

\begin{lemma} \label{lemma: sclosed sungroup corresponds to sHopfideal}
	Let $G$ be a $\s$-algebraic group. There is a one-to-one correspondence between the $\s$-closed subgroups of $G$ and the $\s$-Hopf ideals in $k\{G\}$. \qed
\end{lemma}

Let us see some examples of $\s$-algebraic groups.

\begin{ex}
	Let $V$ be a finite dimensional $k$-vector space. For every $k$-$\s$-algebra $R$ let $\Gl_V(R)$ denote the group of $R$-linear automorphisms of $V\otimes_k R$. Then $\Gl_V$ is naturally a functor from $k$-$\s$-$\alg$ to $\groups$. Indeed, $\Gl_V$ is a $\s$-algebraic group: By choosing a basis $v_1,\ldots,v_n$ of $V$, we can identify $\Gl_V(R)$ with $\Gl_n(R)$ and $\Gl_V$ is represented by
	$$k\{\Gl_n\}=k\{x_{ij},\tfrac{1}{\det(x)}\}.$$
\end{ex}
More generally:
\begin{ex}
	Every algebraic group over $k$ can be interpreted as a $\s$-algebraic group over $k$. Here, as throughout the text, by an algebraic group over $k$ we mean an affine group scheme of finite type over $k$. Indeed, let $\G$ be an algebraic group over $k$ and as in Section \ref{subsec: Zariski closures}, let $R^\sharp$ denote the $k$-algebra obtained from the $k$-$\s$-algebra $R$ by forgetting \nolinebreak[4] $\s$. Then
	$$R\rightsquigarrow \G(R^\sharp)$$
	is a functor from $k$-$\s$-$\alg$ to $\groups$, i.e., $[\s]_k\G$ is a $\s$-algebraic group.
\end{ex}
We will often write $\G$ instead of $[\s]_k\G$. In particular, by a $\s$-closed subgroup of $\G$ we mean a $\s$-closed subgroup of $[\s]_k\G$ and we write $k\{\G\}$ for $k\{[\s]_k\G\}=[\s]_k k[\G]$.

\begin{ex}
	Let $0\leq \alpha_1<\cdots<\alpha_n$ and $1\leq\beta_1,\ldots,\beta_n$ be integers. We can define a $\s$-closed subgroup $G$ of the multiplicative group $\Gm$ by $$G(R)=\{g\in R^\times|\ \s^{\alpha_1}(g)^{\beta_1}\cdots\s^{\alpha_n}(g)^{\beta_n}=1\}\leq\Gm(R)$$
	for any $k$-$\s$-algebra $R$. Similarly, we can define an endomorphism $\f\colon \Gm\to\Gm$ of the multiplicative group (considered as a $\s$-algebraic group) by $$\f_R\colon \Gm(R)\to \Gm(R),\ g\mapsto \s^{\alpha_1}(g)^{\beta_1}\cdots\s^{\alpha_n}(g)^{\beta_n}$$
	for any \ks-algebra $R$.
\end{ex}

\begin{ex} \label{ex: linear}
	Every homogeneous linear difference equation $$\s^n(y)+\lambda_{n-1}\s^{n-1}(y)+\cdots+\lambda_0y=0$$ over $k$ defines a $\s$-closed subgroup $G$ of the additive group $\Ga$, i.e.,
	$$G(R)=\{g\in R| \ \s^n(g)+\lambda_{n-1}\s^{n-1}(g)+\cdots+\lambda_0g=0\}\leq \Ga(R)$$
	for any $k$-$\s$-algebra $R$. We can define an endomorphism $\f\colon \Ga\to\Ga$ of the additive group (considered as a $\s$-algebraic group) by $$\f_R\colon \Ga(R)\to \Ga(R),\ g\mapsto \s^n(g)+\lambda_{n-1}\s^{n-1}(g)+\cdots+\lambda_0g $$
	for any \ks-algebra $R$.
\end{ex}

\begin{ex} \label{ex: unitary group}
	The equations of the unitary group define a $\s$-closed subgroup of the general linear group:
	$$G(R)=\{g\in\Gl_n(R)|\ g\s(g)^{\operatorname{T}}=\s(g)^{\operatorname{T}}g=I_n\}\leq\Gl_n(R)$$
	for any $k$-$\s$-algebra $R$.
\end{ex}

\begin{ex} \label{ex: constant Sln}
	For $m\geq 1$ we can define a $\s$-closed subgroup $G$ of $\Sl_n$ by
	$$G(R)=\{g\in\Sl_n(R)|\ \s^m(g)=g\}\leq\Sl_n(R)$$
	for any \ks-algebra $R$.
\end{ex}

Examples \ref{ex: unitary group} and \ref{ex: constant Sln} can be generalized as follows:

\begin{ex} \label{ex: general constr}
	Let $m\geq 1$ be an integer, $\G$ an algebraic group over $k$ and $\f\colon \G\to{}^{\s^m\!}\G$ a morphism of algebraic groups over $k$.
	There is a morphism $\s^m\colon \G\to{}^{\s^m\!}\G$ of $\s$-algebraic groups over $k$, which, in terms of equations is given by applying $\s^m$ to the coordinates. In terms of $k$-$\s$-algebras this morphism can be described as ${}^{\s^m\!}(k\{\G\})=k\{\G\}\otimes_k k\to k\{\G\},\ f\otimes\lambda\mapsto\s^m(f)\lambda$. We can define a $\s$-closed subgroup $G$ of $\G$ by
	$$G(R)=\{g\in\G(R)|\ \s^m(g)=\f(g)\}\leq\G(R)$$
	for any $k$-$\s$-algebra $R$. Note that Example \ref{ex: unitary group} ($\G=\Gl_n$, $m=1$, $\f(g)=(g^{-1})^{\operatorname{T}}$) and Example \ref{ex: constant Sln} ($\G=\Sl_n$, $\f(g)=g$) are special cases of this construction.
\end{ex}

\begin{ex} \label{ex: new}
	For a \ks-algebra $R$ let $G(R)=R\times R^\times$. A straight forward computation shows that setting
	$$(n_1,h_1)\cdot (n_2,h_2)=(n_1+h_1\s(h_1)^{-1}n_2,h_1h_2)$$
	for $(n_1,h_1),(n_2,h_2)\in G(R)$ defines a group structure on $G(R)$. Thus $G$ is naturally a \mbox{$\s$-algebraic} group.
\end{ex}

\begin{ex} \label{ex: Frobenius algebraic group}
	Let $k$ be a field of positive characteristic $p$ and let $q$ be a power of $p$. Consider $k$ as a $\s$-field via the Frobenius, i.e., $\s(\lambda)=\lambda^q$ for $\lambda\in k$. An algebraic group $\G$ over $k$ can be considered as a $\s$-algebraic group over $k$: We turn $k[\G]$ into a $k$-$\s$-algebra by setting $\s(f)=f^q$ for $f\in k[\G]$. Then clearly $k[\G]$ is a $k$-$\s$-Hopf algebra.
\end{ex}

A finite group can be interpreted an as algebraic group, usually called a finite constant group. See e.g., \cite[2.3]{Milne:AlgebraicGroupsTheTheoryOfGroupSchemesOfFiniteTypeOverAField} or \cite[Section 2.3]{Waterhouse:IntroductiontoAffineGroupSchemes}. The following example is a difference analog of this construction.

\begin{ex} \label{ex: split finite}
	Let $\mathsf{G}$ be a finite group and $\s\colon \mathsf{G}\to \mathsf{G}$ a group endomorphism.
	We define a functor $G$ from $\ksalg$ to $\groups$ as follows: For a \ks-algebra $R$, let $G(R)$ denote the set of all locally constant maps $f\colon \spec(R)\to\mathsf{G}$ such that
	$$
	\xymatrix{
		\spec(R) \ar_\Sigma[d] \ar^-f[r] & \mathsf{G} \ar^\sigma[d] \\
		\spec(R) \ar^-f[r] & \mathsf{G}		
}
	$$
	commutes, where $\Sigma(\p)=\s^{-1}(\p)$ (for $\s\colon R\to R$). The group structure on $G(R)$ is given by pointwise multiplication. Note that the definition of $G(R)$ is functorial in $R$ because a morphism $R_1\to R_2$ of \ks-algebras induces a continuous map $\spec(R_2)\to \spec(R_1)$ such that 
	$$
	\xymatrix{
	\spec(R_2) \ar[r] \ar_\Sigma[d] & \spec(R_1) \ar^\Sigma[d] \\
	\spec(R_2) \ar[r] & \spec(R_1)	
	}
	$$
	commutes. To show that $G$ is a $\s$-algebraic group, consider the $k$-algebra $k^\mathsf{G}$ of all maps from $\mathsf{G}$ to $k$ (with pointwise addition and multiplication). For $g\in \mathsf{G}$ let $\p_g$ denote the prime ideal of $k^\mathsf{G}$ consisting of all functions that map $g$ to $0$. Note that every prime ideal of $k^\mathsf{G}$ is of the form $\p_g$ for a unique $g\in\mathsf{G}$.
	
	For a $k$-algebra $R$, the map that associates to a morphism $\psi\colon k^\mathsf{G}\to R$ of $k$-algebras, the map $\widetilde{\psi}\colon\spec(R)\to \mathsf{G},\ \p\mapsto g $, where $g$ is the unique $g\in\mathsf{G}$ such that $\psi^{-1}(\p)=\p_g$ is a bijection from the set of $k$-algebra morphisms $k^\mathsf{G}\to R$ to the set of locally constant maps $\spec(R)\to \mathsf{G}$ (\cite[Chapter I, \S 1, 6.9]{DemazureGabriel:GroupesAlgebriques}).
		
	We define the structure of a \ks-algebra on $k^\mathsf{G}$ by $\s(h(g))=\s(h(\s(g))$ for $g\in \mathsf{G}$ and $h\colon \mathsf{G}\to k$. We claim that $\Hom(k^\mathsf{G},R)\simeq G(R)$ for every \ks-algebra $R$.
	First, note that the bijection $\mathsf{G}\to \spec(k^\mathsf{G}),\ g\mapsto \p_g$ is such that 
	$$
	\xymatrix{
	\mathsf{G} \ar[r] \ar_\s[d] & \spec(k^\mathsf{G}) \ar^\Sigma[d] \\
	\mathsf{G} \ar[r] & \spec(k^\mathsf{G})
}
	$$
	commutes. Since 
	$$ 
	\xymatrix{
		\spec(R) \ar[r] \ar_\Sigma[d] & \spec(k^\mathsf{G})  \ar^\Sigma[d] \\
		\spec(R) \ar[r] & \spec(k^\mathsf{G})
	}
	$$
	commutes for every $\psi\in\Hom(k^\mathsf{G},R)$, we see that the map $\Hom(k^\mathsf{G},R)\to G(R),\ \psi\mapsto \widetilde{\psi}$ is well-defined. It remains to show that a morphism $\psi\colon k^\mathsf{G}\to R$ of $k$-algebras is a morphism of $\s$-rings if $\widetilde{\psi}$ lies $G(R)$.
	In other words, we need to show that
	\begin{equation} \label{eq: diag for example}
	\xymatrix{
		 k^\mathsf{G} \ar^\psi[r] \ar_\s[d] & R \ar^\s[d] \\
		 k^\mathsf{G} \ar^\psi[r] & R
		}
	 \end{equation}
	commutes, if the dual diagram
	$$
	\xymatrix{
		\spec(k^\mathsf{G})  & \spec(R) \ar[l]  \\
		\spec(k^\mathsf{G}) \ar^{\Sigma}[u]  & \spec(R) \ar[l] \ar_{\Sigma}[u]
	}
	$$
	commutes.
	As in Section \ref{subsec: Zariski closures}, let ${\hs R}=R\otimes_k k$ denote the $k$-algebra obtained from the $k$-algebra $R$ by base change via $\s\colon k\to k$. Then $\widetilde{\s}\colon{\hs R}\to R,\ r\otimes \lambda\mapsto \s(r)\lambda$ is a morphism of $k$-algebras. We define $\widetilde{\s}\colon{\hs (k^{\mathsf{G}})}\to k^{\mathsf{G}}$ similarly. Then diagram (\ref{eq: diag for example}) commutes if and only if
	the diagram	\begin{equation} \label{eq; diag for example 2}
	\xymatrix{
		{\hs (k^\mathsf{G})} \ar^{\hs \psi}[r] \ar_{\widetilde{\s}}[d] & {\hs R \ar^{\widetilde{\s}}[d]} \\
		k^\mathsf{G} \ar^\psi[r] & R
	}
	\end{equation}
	of $k$-algebras commutes. Since two morphisms of $k$-algebras ${\hs (k^\mathsf{G})}\to R$ agree if they induce the same map $\spec(R)\to\spec({\hs (k^\mathsf{G})})$, it suffices to show that the diagram
		\begin{equation} \label{eq: diag for example 3}
	\xymatrix{
		\spec({\hs (k^\mathsf{G})}) & \spec({\hs R}) \ar[l] \\
		\spec(k^\mathsf{G}) \ar[u]  & \spec(R) \ar[l] \ar[u]
	}
	\end{equation}
	dual to (\ref{eq; diag for example 2}) commutes.
	Let $i\colon R\to {\hs R},\ r\mapsto r\otimes 1$, then $R\xrightarrow{i}{\hs R}\xrightarrow{\widetilde{\s}}R$ equals $\sigma$ and similarly for $k^\mathsf{G}$ in place of $R$. In the diagram
	$$ 
	\xymatrix{
		\spec(k^\mathsf{G}) & \spec(R) \ar[l] \\
		\spec({\hs(k^\mathsf{G})}) \ar[u] & \spec({\hs R}) \ar[l] \ar[u]	\\
		\spec(k^\mathsf{G}) \ar[u]  \ar@/^3pc/^\Sigma[uu] & \spec(R) \ar[l] \ar[u] \ar@/_3pc/_\Sigma[uu]
}
	$$
	the upper square, the outer square and the left and right triangles commute. Moreover, the map $\spec({\hs(k^\mathsf{G})})\to \spec(k^\mathsf{G})$ is bijective. It thus follows that diagram (\ref{eq: diag for example 3}) commutes. So $k^\mathsf{G}$ represents $G$.
	
	The Hopf algebra structure on $k^\mathsf{G}$ agrees with the one described in \cite[Section 2.3]{Waterhouse:IntroductiontoAffineGroupSchemes}.
\end{ex}
%
%

To show that any $\s$-algebraic group is isomorphic to a $\s$-closed subgroup of some general linear group, we need a lemma on Hopf algebras:

\begin{lemma} \label{lemma: Hopf adjunction}
	Let $A$ be a $k$-Hopf algebra, $B$ a \ks-Hopf algebra and $A\to B$ a morphism of $k$-Hopf algebras. Then the morphism $[\s]_kA\to B$ of \ks-algebras induced by Lemma~\ref{lemma: universal property of sA} is a morphism of $k$-Hopf algebras. 
\end{lemma}
\begin{proof}
	The comultiplication of $[\s]_k A$ is $[\s]_k(\Delta)\colon [\s]_kA\to[\s]_k(A\otimes_k A)=[\s]_k A\otimes_k [\s]_k A$ where $\Delta\colon A\to A\otimes_k A$ is the comultiplication of $A$.
	All the subdiagrams with empty interior of 
	$$
	\xymatrix{
		[\s]_kA \ar[r] \ar@/_20pt/[dd] & [\s]_k(A\otimes_k A) \ar@/^25pt/[dd] \\
		A \ar[u] \ar[d] \ar[r] & A\otimes_k A \ar[u] \ar[d] \\
		B \ar[r] & B\otimes_k B	
	}
	$$
	commute. We have to show that the outer diagram commutes. But this follows from the uniqueness statement in Lemma \ref{lemma: universal property of sA}. A similar argument applies for the counit.
\end{proof}
Similar to (affine) algebraic groups, $\s$-algebraic groups can be linearized:
\begin{prop} \label{prop: linearization}
	Let $G$ be a $\s$-algebraic group over $k$. Then there exists a finite dimensional $k$-vector space $V$ and a $\s$-closed embedding
	$G\hookrightarrow \Gl_V$. In particular, $G$ is isomorphic to a $\s$-closed subgroup of $\Gl_{n}$ for some $n\geq 1$.
\end{prop}
\begin{proof}
	Assume that $f_1,\ldots,f_m\in k\{G\}$ generate $k\{G\}$ as a $k$-$\s$-algebra. By \cite[Section 3.3, p. 24]{Waterhouse:IntroductiontoAffineGroupSchemes} there exists a Hopf subalgebra $A$ of $k\{G\}$ that contains $f_1,\ldots,f_m$ and is finitely generated as a $k$-algebra. By \cite[Section 3.4, p. 25]{Waterhouse:IntroductiontoAffineGroupSchemes} there exists an integer $n\geq 1$ and a surjective morphism $k[\Gl_{n}]\to A$ of Hopf algebras. By Lemma  \ref{lemma: universal property of sA} the morphism $k[\Gl_{n}]\to A\hookrightarrow k\{G\}$ of $k$-algebras extends to a morphism
	$k\{\Gl_{n}\}\to  k\{G\}$ of $k$-$\s$-algebras. Indeed, this is a morphism of $k$-$\s$-Hopf algebras (Lemma \ref{lemma: Hopf adjunction}) and since $f_1,\ldots,f_m$ lie in the image, it is surjective. Now the claim follows from Lemma \ref{lemma: sclosed embedding}.
\end{proof}


\section{Zariski closures of difference algebraic groups and the growth group} \label{sec: Zariski closures of difference algebraic groups and the growth group}

%

In this section we show that the eventual growth of the Zariski closures of a $\s$-closed subgroup $G$ of an algebraic group $\G$ is governed by an algebraic group, called the \emph{growth group} of $G$ (with respect to the $\s$-closed embedding $G\hookrightarrow\G$). This result is key for obtaining the finiteness results in the next section. We also introduce the $\s$-dimension and the order of a $\s$-algebraic group.

\medskip

Let $k$ be a $\s$-field and $\G$ an algebraic group over $k$. Then, for ever $i\geq 0$, ${^{\sigma^i}\!\G}$ and $\G[i]=\G\times{^{\sigma}\!\G}\times\cdots\times{^{\sigma^i}\!\G}$ (see Section \ref{subsec: Zariski closures}) are naturally algebraic groups over $k$. 
For every $i\geq 1$ the maps $\pi_i\colon\G[i]\to\G[i-1]$ and $\s_i\colon\G[i]\to{\hs(\G[i-1])}$ are morphisms of algebraic groups over $k$.


Let $G$ be a $\s$-closed subgroup of $\G$. For every $i\geq 0$ we have an inclusion $k[\G[i]]\subseteq k\{\G\}$ of Hopf algebras. Since $\I(G)\subseteq k\{\G\}$ is a Hopf ideal, $\I(G)\cap k[\G[i]]$ is a Hopf ideal of $k[\G[i]]$. So the $i$-th order Zariski closure $G[i]$ of $G$ in $\G$ is a closed subgroup of $\G[i]$. The maps $\pi_i\colon G[i]\to G[i-1]$ and $\s_i\colon G[i]\to{\hs(G[i-1])}$ are morphisms of algebraic groups over $k$ and form a commutative diagram:
\begin{equation} \label{eqn: comm diag pi si}
\xymatrix{
	G[i] \ar_{\s_i}[d] \ar^-{\pi_i}[r] & G[i-1]  \ar^{\s_{i-1}}[d] \\
	{}^{\s\!}(G[i-1]) \ar^-{{}^{\s\!}\pi_{i-1}}[r] & {}^{\s\!}(G[i-2]) \\
}
\end{equation}
For $i\geq 1$ we set
$$\G_i=\ker(\pi_i)\leq G[i].$$ We also set $\G_0=G[0]$. Because of (\ref{eqn: comm diag pi si}) we have induced morphisms $\s_i\colon \G_i\to{\hs(\G_{i-1})}$ of algebraic groups over $k$.

\begin{prop} \label{prop: existence of growth group}
	For every $i\geq 1$ the map $\s_i\colon \G_i\to{\hs(\G_{i-1})}$ is a closed embedding and there exists an integer $m\geq 1$ such that
	$\s_i\colon \G_i\to{\hs(\G_{i-1})}$ is an isomorphism for every $i\geq m$.
\end{prop}
\begin{proof}
	Let us start by describing $\s_i$ in algebraic terms. Assume that $A=\{a_1,\ldots,a_n\}$ generates $k[\G]$ as a $k$-algebra and let $\overline{A}$ denote the image of $A$ in $k\{G\}$. So $$k[\G]=k[A]\subseteq k[A,\s(A),\ldots]=k\{\G\}$$ and $k[G[i]]=k[\overline{A},\ldots,\s^i(\overline{A})]$. The morphism $\s_i\colon G[i]\to{\hs(G[i-1])}$ corresponds to the map
	$${\hs\big(k[\overline{A},\ldots,\s^{i-1}(\overline{A})]\big)}\longrightarrow k[\overline{A},\ldots,\s^i(\overline{A})],\ f\otimes\lambda\mapsto \s(f)\lambda$$ and the morphism $\s_i\colon \G_i\to{\hs(\G_{i-1})}$ corresponds to the map
	\begin{align*}{\hs\big(k[\overline{A},\ldots,\s^{i-1}(\overline{A})]\otimes_{k[\overline{A},\ldots,\s^{i-2}(\overline{A})]} k\big)} & \longrightarrow k[\overline{A},\ldots,\s^{i}(\overline{A})]\otimes_{k[\overline{A},\ldots,\s^{i-1}(\overline{A})]} k \\  (f\otimes\lambda)\otimes\mu & \longmapsto \s(f)\otimes\s(\lambda)\mu
	\end{align*}
	where the tensor products are formed in virtue of the counit $\varepsilon\colon k\{G\}\to k$. Since $k[\overline{A},\ldots,\s^{i}(\overline{A})]\otimes_{k[\overline{A},\ldots,\s^{i-1}(\overline{A})]}k$ is generated as a $k$-algebra by the image of $\s^i(\overline{A})$, the above map is clearly surjective. Thus $\s_i\colon \G_i\to{\hs(\G_{i-1})}$ is a closed embedding.
	
	To prove the second claim of the proposition, let us first assume that $k$ is inversive. The map \begin{align*} \psi_i\colon k[\overline{A},\ldots,\s^{i-1}(\overline{A})]\otimes_{k[\overline{A},\ldots,\s^{i-2}(\overline{A})]} k & \longrightarrow k[\overline{A},\ldots,\s^{i}(\overline{A})]\otimes_{k[\overline{A},\ldots,\s^{i-1}(\overline{A})]} k\\  f\otimes\lambda& \longmapsto \s(f)\otimes\s(\lambda)
	\end{align*} is a morphism of rings but it is not a morphism of $k$-algebras. However, since $k$ is inversive, it is surjective. We thus have a descending chain of closed subschemes
	$$\G_0\hookleftarrow\G_1\hookleftarrow\G_2\cdots.$$ Since $\G_0$ is of finite type over $k$, this sequence must stabilize. That is, there exists an integer $m\geq 1$ such that $\psi_i$ is bijective for every $i\geq m$. Since $k$ is inversive, $$k[\G_i]\to{\hs(k[\G_i])},\ f\mapsto f\otimes 1$$ is bijective. It follows that for $i\geq m$, the morphism ${\hs(k[\G_{i-1}])}\to k[\G_i]$ dual to $\s_i$ is an isomorphism, since it can be obtained as the composition ${\hs(k[\G_{i-1}])}\to k[\G_{i-1}]\xrightarrow{\psi_i} k[\G_i]$ of two bijective maps. This proves the proposition for $k$ inversive.
	
	The general case can be reduced to the inversive case: Let $k^*$ denote the inversive closure of $k$ (\cite[Definition 2.1.6]{Levin:difference}). So, in particular, $k^*$ is an inversive $\s$-field extension of $k$. The formation of Zariski closures and of the $\G_i$ is compatible with base change. It therefore follows from the inversive case, that there exists an integer $m\geq 1$ such that for every $i\geq m$ the morphism $\s_i\colon \G_i\to{\hs(\G_{i-1})}$ becomes an isomorphism after base change from $k$ to $k^*$. But then already $\s_i$ must be an isomorphism for $i\geq m$.
\end{proof}

Proposition \ref{prop: existence of growth group} allows us to associate to the inclusion $G\leq\G$ an algebraic group, that measures the (eventual) growth of the Zariski closures $G[i]$ of $G$ in $\G$.

\begin{defi}
	Let $\G$ be an algebraic group and $G\leq\G$ a $\s$-closed subgroup. For $i\geq 1$ let $\G_i$ denote the kernel of the projection $\pi_i\colon G[i]\to G[i-1]$ between the Zariski closures of $G$ in $\G$. Let $m\geq 0$ denote the smallest integer such that $\s_i\colon \G_i\to{\hs(\G_{i-1})}$ is an isomorphism for every $i>m$. Then $\G_m$ is called the \emph{growth group of $G$}\index{growth group of $G$} with respect to the $\s$-closed embedding $G\hookrightarrow\G$.
\end{defi}

\begin{ex} \label{ex: ssubgroup of Gm}
	Let $0\leq \alpha_1<\cdots<\alpha_n$ and $1\leq\beta_1,\ldots,\beta_n$ be integers and $G\leq\Gm$ the $\s$-closed subgroup of the multiplicative group $\Gm$ given by $$G(R)=\{g\in R^\times|\ \s^{\alpha_1}(g)^{\beta_1}\cdots\s^{\alpha_n}(g)^{\beta_n}=1\}$$
	for every $k$-$\s$-algebra $R$. Then the growth group of $G$ with respect to the given embedding $G\hookrightarrow\Gm$ is $\mu_{\beta_n}$, where $\mu_{\beta_n}(T)=\{g\in T^\times|\ g^{\beta_n}=1\}$ for any $k$-algebra $T$.
\end{ex}

\begin{ex} \label{ex: growth group of linear}
	Let $G$ be a $\s$-closed subgroup of the additive group defined by a linear difference equation. (Cf. Example \ref{ex: linear}.) Then the growth group with respect to the given embedding $G\iar\Ga$ is trivial.
\end{ex}

\begin{ex} \label{ex: growth group of algebraic group}
	Let $\G$ be an algebraic group. The growth group with respect to the tautological embedding $G=[\s]_k\G\hookrightarrow \G$ is $\G$ itself.
\end{ex}

The following example shows that the growth group does indeed depend on the embedding $G\hookrightarrow\G$ and therefore is not an invariant of $G$.
\begin{ex}
	Let $G=\Gm$ (considered as a $\s$-algebraic group). For $n\geq 1$ the $\s$-closed embedding
	$$G\to\Gm^2,\ g\mapsto (g,\s(g)^n)$$ identifies $G$ with the $\s$-closed subgroup $G\leq\Gm^2$ given by
	$$G(R)=\{(g_1,g_2)\in\Gm(R)^2|\ \s(g_1)^n=g_2\}$$
	for any $k$-$\s$-algebra $R$. The growth group of $G$ with respect to the embedding $G\hookrightarrow\Gm^2$ is $\mu_n\times\Gm$.
\end{ex}

Even though the growth group itself does depend on the chosen embedding $G\hookrightarrow\G$, it carries some information that only depends on $G$. We will show below that the dimension of the growth group does not depend on the embedding $G\hookrightarrow\G$. In Section \ref{sec: limit degree} we will see that also the size of the growth group is independent of the chosen embedding.

The following theorem can be seen as a group theoretic analog of the classical theorem on the existence of the so-called dimension polynomial of an extension of $\s$-fields (\cite[Theorem 4.2.1]{Levin:difference}). See also \cite[Lemma 4.21]{Hrushovski:elementarytheoryoffrobenius} (but note that we do not need any integrality assumptions).

\begin{theo} \label{theo: existence sdim}
	Let $\G$ be an algebraic group and $G\leq\G$ a $\s$-closed subgroup. For $i\geq 0$ let $d_i=\dim(G[i])$ denote the dimension of the $i$-th order Zariski closure of $G$ in $\G$. Then there exist integers $d,e\geq0$ such that
	$$d_i=d(i+1)+e \text{ for } i\gg 0.$$
	The integer $d$ only depends on $G$ and not on the choice of $\G$ and the $\s$-closed embedding $G\hookrightarrow \G$. If $d=0$, the integer $e$ only depends on $G$ and not on the choice of $\G$ and the $\s$-closed embedding $G\hookrightarrow \G$.
\end{theo}
\begin{proof}
	Let $\G_i$ denote the kernel of $G[i]\to G[i-1]$. Then $\dim(G[i])=\dim(G[i-1])+\dim(\G_i)$. Let $m\geq 0$ denote the smallest integer such that $\s_i\colon \G_i\to{\hs(\G_{i-1})}$ is an isomorphism for every $i>m$ (Proposition \ref{prop: existence of growth group}). It follows that for $i\gg 0$
	\begin{align*}\dim(G[i]) &=\dim(\G_i)+\cdots+\dim(\G_0)=\\
	& = (i-m+1)\dim(\G_m)+\dim(\G_{m-1})+\cdots+\dim(\G_0)=\\
	& = (i+1)\dim(\G_m)+\dim(\G_{m-1})+\cdots+\dim(\G_0)-m\dim(\G_m)=\\
	& = d(i+1)+e.
	\end{align*}
	By Proposition \ref{prop: existence of growth group} we have $e\geq 0$.
	
	Let us now show that $d=\dim(\G_m)$ is independent of the chosen embedding $G\hookrightarrow\G$. So let $G\hookrightarrow \G'$ be another $\s$-closed embedding. Let $G[i]'$ denote the $i$-th order Zariski closure of $G$ in $\G'$ and let $d_i', d',e'$ have the analogous meaning. We have to show that $d=d'$.
	
	Let $A$ be a finite subset of $k\{G\}$ that generates $k[G[0]]\subseteq k\{G\}$ as a $k$-algebra. Similarly, let $A'$ be a finite subset of $k\{G\}$ that generates $k[G[0]']\subseteq k\{G\}$ as a $k$-algebra. Then $k[G[i]]=k[A,\ldots,\s^i(A)]$ and there exists an integer $n\geq 0$ such that $A'$ lies in $k[G[n]]$. Then $$k[G[i]']=k[A',\ldots,\s^i(A')]\subseteq k[A,\ldots,\s^{n+i}(A)]=k[G[n+i]].$$ Therefore $d_i'\leq d_{n+i}$ and for $i\gg 0$ we have
	$d'(i+1)+e'\leq d(n+i+1)+e$. Letting $i$ tend to infinity we find $d'\leq d$. By symmetry, $d'=d$.
	
	Let us now assume that $d=0$. We have to show that $e$ does not depend on the choice of the embedding $G\hookrightarrow\G$. To do this we will show that
	\begin{equation} \label{eqn: order}
	e=\max\big\{\dim(R)|\ R \text{ is a finitely generated $k$-subalgebra of $k\{G\}$}\big\}.
	\end{equation}
	
	For $i\gg0$ the finitely generated $k$-subalgebra $k[G[i]]$ of $k\{G\}$ has dimension $e$. Conversely, if $R$ is a finitely generated $k$-subalgebra of $k\{G\}$, then $R$ is contained in some $k[G[i]]$ and therefore $\dim(R)\leq e$. This proves (\ref{eqn: order}).
\end{proof}

\begin{defi}
	Let $G$ be a $\s$-algebraic group. The integer $d\geq 0$ defined in Theorem \ref{theo: existence sdim} above is called the \emph{$\s$-dimension of $G$}\index{$\s$-dimension of $G$} and denoted by $\sdim(G).$
	
	If $\sdim(G)=0$, the integer $e\geq 0$ defined in Theorem~\ref{theo: existence sdim} is called the \emph{order of $G$}\index{order of $G$} and denoted by $\ord(G).$ If $G$ has positive $\s$-dimension the order of $G$ is defined to be infinity.
\end{defi}

The order is equivalent to the \emph{total dimension} introduced in \cite{Hrushovski:elementarytheoryoffrobenius} in a somewhat different setting. We have chosen to stick to the more traditional naming from \cite{Levin:difference} and \cite{Cohn:difference}.

\begin{ex}
	Let $n\in\nn$ and let $G$ be the $\s$-closed subgroup of $\Ga^2$ given by
	$$G(R)=\left\{(g_1,g_2)\in R^2|\ \s^n(g_1)=g_2\right\}$$
	for any \ks-algebra $R$. For $i\geq 0$ let $G[i]$ denote the $i$-th order Zariski closure of $G$ in $\Ga$. As $k\{G\}=k[y_1,\ldots,\s^{n-1}(y_1),y_2,\s(y_2),\ldots]$ we see that
	$\dim(G[i])=(i+1)+n$ for $i\gg 0$. In particular, $G$ has $\s$-dimension one.
\end{ex}

\begin{ex}
	Let $\G$ be an algebraic group. Then $\sdim([\s]_k\G)=\dim(\G)$. Indeed, if we tautologically consider $G=[\s]_k\G$ as a $\s$-closed subgroup of $\G$, then \mbox{$G[i]=\G\times\cdots\times{^{\sigma^i}\!\G}$} and so $\dim(G[i])=\dim(\G)(i+1)$ for every $i\geq 0$. This also shows that either $\ord(G)=\infty$ (if $\dim(\G)>0$) or $\ord(G)=0$ (if $\dim(\G)=0$).
\end{ex}
The following example also motivates the naming ``order''.
\begin{ex}
	Let $f=\s^n(y)+\lambda_{n-1}\s^{n-1}(y)+\cdots+\lambda_0y$ be a linear difference equation and $G$ the $\s$-closed subgroup of $\Ga$ defined by $f$. Then
	$\ord(G)=n$, i.e., the order of $G$ equals the order of $f$.
\end{ex}

\begin{ex}
	The $\s$-algebraic group from Example \ref{ex: ssubgroup of Gm} has $\s$-dimension zero and order $\alpha_n$.
\end{ex}

\begin{ex}
	Let $G$ be the $\s$-algebraic group from Example \ref{ex: general constr}. Then $k\{G\}=k[\G\times\cdots\times {}^{\s^{m-1}}\!\G]$ and therefore $\sdim(G)=0$ and $\ord(G)=m\dim(\G)$.
\end{ex}

\begin{ex}
	The $\s$-algebraic group from Example \ref{ex: Frobenius algebraic group} has $\s$-dimension zero and order equal to the dimension of $\G$.
\end{ex}

\begin{ex}
	The $\s$-algebraic group from Example \ref{ex: split finite} has $\s$-dimension zero and order zero.
\end{ex}
From the proof of Theorem \ref{theo: existence sdim}, we immediately obtain:
\begin{cor} \label{cor: sdim equals dim of growth group}
	Let $G$ be a $\s$-algebraic group. Then the dimension of the growth group of $G$ with respect to some $\s$-closed embedding $G\hookrightarrow\G$ of $G$ into some algebraic group $\G$ equals the $\s$-dimension of $G$. In particular, the dimension of the growth group does not depend on the choice of the $\s$-closed embedding $G\hookrightarrow\G$. \qed
\end{cor}



We will show that our notions of dimension and order generalize the classical notions. (See page 394 in \cite{Levin:difference}.)
Classically, the $\s$-dimension of a $\s$-variety $X$ is defined by means of the $\s$-transcendence degree, which however only makes sense if $k\{X\}$ is an integral domain with $\s\colon k\{X\}\to k\{X\}$ injective.
The \emph{$\s$-transcendence degree} of a $\s$-field extension $K/k$ is the largest integer $n\geq 1$ such that the $\s$-polynomial ring $k\{y_1,\ldots,y_n\}$ may be embedded into $K$. (If no such integer exists the $\s$-transcendence degree is infinite.) See Section 4.1 in \cite{Levin:difference} for more details on the $\s$-transcendence degree. 
\begin{prop} \label{prop: characterize sdim}
	Let $G$ be a $\s$-algebraic group such that $k\{G\}$ is a integral domain with $\s\colon k\{G\}\to k\{G\}$ injective. Then the $\s$-dimension of $G$ equals the $\s$-transcendence degree of the field of fractions of $k\{G\}$ over $k$.
	
	If $G$ is a $\s$-algebraic group such that $k\{G\}$ is an integral domain, then the order of $G$ equals the transcendence degree of the field of fractions of $k\{G\}$ over $k$.
\end{prop}
\begin{proof}
	Let us fix a $\s$-closed embedding $G\hookrightarrow\G$. Assume that the finite set $A$ generates $k[G[0]]\subseteq k\{G\}$ as a $k$-algebra. Let $K$ denote the field of fractions of $k\{G\}$.
	Since $\s\colon k\{G\}\to k\{G\}$ is injective, $K$ is naturally a $\s$-field extension of $k$ and $A$ generates $K$ as a $\s$-field extension of $k$. Since $\trdeg(k(A,\ldots,\s^i(A))/k)=\dim(G[i])$, we see that
	$d(t+1)+e$ (with $d$ and $e$ as in Theorem \ref{theo: existence sdim}) is the difference dimension polynomial (Definition 4.2.2 in \cite{Levin:difference}) of the $\s$-field extension $K/k$ associated with $A$. It follows from Theorem 4.2.1 (iii) in \cite{Levin:difference} that $\sdim(G)=d=\strdeg(K/k)$.
	
	For the second claim, let again $K$ denote the field of fractions of $k\{G\}$. If $\sdim(G)>0$, then it is clear from Theorem \ref{theo: existence sdim} that the transcendence degree of $K$ over $k$ is infinite. So $\ord(G)=\trdeg(K/k)$ in this case. If $\sdim(G)=0$, the claim follows from equation (\ref{eqn: order}) above.
\end{proof}

\begin{ex}
	Let $G$ be the $\s$-algebraic group from Example \ref{ex: new}. Then the coordinate ring of $G$ is $k\{G\}=k\{y\}\otimes_k k\{x,x^{-1}\}$. The field of fractions of $k\{G\}$ is the field of fractions of a $\s$-polynomial ring in two $\s$-variables. Therefore the field of fractions of $k\{G\}$ has $\s$-transcendence degree two over $k$. (Cf. \cite[Prop. 4.1.6, p. 248]{Levin:difference}.) Thus $\sdim(G)=2$ by Proposition~\ref{prop: characterize sdim}.
\end{ex}

%
%
%
%

\section{Two finiteness theorems} \label{sec: two finiteness theorems}

There is no direct difference analog of Hilbert's basis theorem: There exist infinite strictly ascending chains of difference ideals in the $\s$-polynomial ring $k\{y_1,\ldots,y_n\}$. (See \cite[Example 3, p. 73]{Cohn:difference}.) The Ritt-Raudenbush basis theorem in difference algebra (\cite{RittRaudenbush:IdealTheoryandAlgebraicDifferenceEquations} or \cite[Theorem 2.5.11]{Levin:difference}) only asserts that every ascending chain of perfect difference ideals in $k\{y_1,\ldots,y_n\}$ is finite. However, in our group theoretic setup the situation is better behaved: Let $A$ be a finitely $\s$-generated $k$-$\s$-Hopf algebra. The first finiteness theorem (Theorem~\ref{theo: finiteness1}) asserts that every $\s$-Hopf ideal of $A$ is finitely generated as a $\s$-ideal. The second finiteness theorem (Theorem \ref{theo: finiteness2}) asserts that every $k$-$\s$-Hopf subalgebra of $A$ is finitely $\s$-generated over $k$. In Section \ref{section: components} we will prove a third finiteness theorem (Theorem \ref{theo: Hrushovskiconj}): The set of minimal prime $\s$-ideals of $A$ is finite.


\begin{theo}  \label{theo: finiteness1}
	Let $H$ be a $\s$-algebraic group and $G\leq H$ a $\s$-closed subgroup. Then the defining ideal $\I(G)\subseteq k\{H\}$ of $G$ is finitely generated as a $\s$-ideal.
\end{theo}
\begin{proof}
	We may embed $H$ as a $\s$-closed subgroup in some algebraic group $\G$. For example, we may choose $\G=\Gl_n$ by Proposition \ref{prop: linearization}.
	If the defining ideal of $G$ in $k\{\G\}$ is finitely \mbox{$\s$-gen}erated, then also the defining ideal of $G$ in $k\{H\}=k\{G\}/\I(H)$ is finitely $\s$-generated. We can therefore assume that $H=\G$.
	
	As in Section \ref{sec: Zariski closures of difference algebraic groups and the growth group}, let $\G_i$ denote the kernel of the projection $\pi_i\colon G[i]\to G[i-1]$ between the Zariski closures of $G$ in $\G$. By Proposition \ref{prop: existence of growth group} there exists an integer $m\geq 1$ such that
	$\s_i\colon \G_i\to {\hs(\G_{i-1})}$ is an isomorphism for every $i> m$.
	To prove the theorem we will show that $\I(G[m])=\I(G)\cap k[\G[m]]$ $\s$-generates $\I(G)\subseteq k\{\G\}=\cup_{i\geq 0}k[\G[i]]$. To do this it is sufficient to show that
	\begin{equation} \label{eqn: sideals of Gi}
	\I(G[i])=\big(\I(G[i-1]),\s(\I(G[i-1]))\big)\subseteq k[\G[i]]
	\end{equation}
	for $i> m$.
	The ideal to the right-hand side of (\ref{eqn: sideals of Gi}) defines an algebraic group
	$$\H_i=(G[i-1]\times{}^{\sigma^i}\!\G)\cap (\G\times{\hs(G[i-1])})\leq \G[i]=\G\times{^{\sigma}\!\G}\times\cdots\times{^{\sigma^i}\!\G}.$$
	Clearly $G[i]\leq\H_i$ and the projection $\pi_i\colon \H_i\to G[i-1]$ is a quotient map, i.e., the dual map is injective.

	The kernel of $\pi_i\colon \H_i\to G[i-1]$ is $1\times {\hs(\G_{i-1})}$. Since $i>m$ we have $1\times {\hs(\G_{i-1})}=\G_i$. Thus the downwards arrows
	in the commutative diagram
	\[\xymatrix{
		G[i] \ar@{^{(}->}[rr] \ar_-{\pi_i}[rd] & & \H_i \ar^-{\pi_i}[ld] \\
		& G[i-1] &
	}
	\]
	are both quotient maps with the same kernel. This implies that $G[i]=\H_i$ and identity (\ref{eqn: sideals of Gi}) is proved.
\end{proof}

We have actually proved a slightly stronger statement which we record for later use.
\begin{cor} \label{cor: finiteness1}
	Let $\G$ be an algebraic group and $G\leq\G$ a $\s$-closed subgroup. For $i\geq 0$ let $G[i]$ denote the $i$-th order Zariski closure of $G$ in $\G$. Then there exists an integer $m\geq 0$ such that $\I(G[i])=(\I(G[i-1],\s(\I(G[i-1]))$ for $i>m$, i.e.,
	$G[i]=(G[i-1]\times{}^{\sigma^i}\!\G)\cap (\G\times{\hs(G[i-1])}).$ \qed
\end{cor}

\begin{cor} \label{cor: descending chain}
	Every descending chain of $\s$-closed subgroups of a $\s$-algebraic group is finite.
\end{cor}
\begin{proof}
	A descending chain $H_1\geq H_2\geq\cdots$ of $\s$-closed subgroups of a $\s$-algebraic group $G$ corresponds to an ascending chain
	$\I(H_1)\subseteq\I(H_2)\subseteq\cdots$ of $\s$-Hopf ideals in $k\{G\}$. By Theorem \ref{theo: finiteness1} the union $\bigcup\I(H_i)$ (which corresponds to the intersection $\bigcap H_i$) is finitely generated as a $\s$-ideal. Thus there exists an integer $n\geq 1$ such that $\bigcup\I(H_i)=\I(H_n)$. Then $H_{i+1}=H_{i}$ for $i\geq n$.
\end{proof}

To prove the second finiteness theorem we need a lemma on $k$-$\s$-Hopf algebras.

\begin{lemma} \label{lemma: directed union for ksHopfalgebras}
	Let $A$ be a $k$-$\s$-Hopf algebra. Then every finite subset of $A$ is contained in a finitely $\s$-generated $k$-$\s$-Hopf subalgebra.
\end{lemma}
\begin{proof}
	By \cite[Section 3.3]{Waterhouse:IntroductiontoAffineGroupSchemes} a finite subset of $A$ is contained in a Hopf subalgebra $B$ that is finitely generated as a $k$-algebra. Then $k\{B\}\subseteq A$ is finitely $\s$-generated over $k$ and since the comultiplication and the antipode are $\s$-morphisms, $k\{B\}$ is a Hopf subalgebra.
\end{proof}

\begin{theo} \label{theo: finiteness2}
	Let $A$ be a $k$-$\s$-Hopf algebra that is finitely $\s$-generated over $k$ and $B\subseteq A$ a $k$-$\s$-Hopf subalgebra. Then $B$ is finitely $\s$-generated over $k$.
\end{theo}
\begin{proof}
	For a Hopf subalgebra $C$ of $A$ let $\m_C\subseteq C$ denote the kernel of the counit $\varepsilon\colon C\to k$. The ideal $(\m_B)\subseteq A$ is a $\s$-Hopf ideal. By Theorem \ref{theo: finiteness1} it is finitely $\s$-generated. So there exists a finite set $F\subseteq \m_B$ such that $[F]=(\m_B)$.
	By Lemma \ref{lemma: directed union for ksHopfalgebras} there exists a finitely $\s$-generated $k$-$\s$-Hopf subalgebra $C$ of $B$ containing $F$.
	Then $(\m_C)=(\m_B)$. By Corollary~3.10 in \cite{Takeuchi:ACorrespondenceBetweenHopfidealsAndSubHopfalgebras} the mapping $C\mapsto (\m_C)$ from Hopf subalgebras to Hopf ideals is injective. Thus $B=C$ is finitely $\s$-generated over $k$.
\end{proof}

%
%
%
%
%
%
%
%

As an application of Corollary \ref{cor: finiteness1}, we will prove a dimension theorem for $\s$-algebraic groups.
 A dimension theorem for differential algebraic groups has been proved in \cite{Sit:TypicalDifferentialDimensionOfTheIntersectionOfLinearDifferentialAlgebraicGroups}.
It is interesting to note that the dimension theorem fails for difference varieties. See \cite[Chapter 8, Section 8]{Cohn:difference}. 

\begin{theo} \label{theo: dimension theorem}
	Let $H_1$ and $H_2$ be $\s$-closed subgroups of a $\s$-algebraic group $G$. Then
	\begin{equation} \label{eqn: sdimintersect}
	\sdim(H_1\cap H_2)+\sdim(G)\geq \sdim(H_1)+\sdim(H_2)\end{equation} and
	\begin{equation} \label{eqn: orderintersect} \ord(H_1\cap H_2)+\ord(G) \geq \ord(H_1)+\ord(H_2).\end{equation}
\end{theo}
\begin{proof}
	Let $\G$ be an algebraic group containing $G$ as a $\s$-closed subgroup. We consider the Zariski closures inside $\G$.
	By Corollary \ref{cor: finiteness1} there exists an integer $m\geq 0$ such that
	$$\I((H_1\cap H_2)[i])=\big(\I((H_1\cap H_2)[i-1]),\s(\I((H_1\cap H_2)[i-1]))\big)$$ for $i>m$. Since $\I(H_1\cap H_2)=\I(H_1)+\I(H_2)$ there exists $n\geq m$ such that $$\I((H_1\cap H_2)[m])\subseteq \I(H_1[n])+\I(H_2[n])=\I(H_1[n]\cap H_2[n]).$$ But then
	\begin{equation} \label{eqn: proof}\I((H_1\cap H_2)[m+i])\subseteq \I(H_1[n+i]\cap H_2[n+i])\end{equation} for every $i\geq 0$ and so
	$$H_1[n+i]\cap H_2[n+i]\leq (H_1\cap H_2)[m+i]\times{{}^{\s^{m+i+1}}\!\G}\times\cdots\times {}^{\s^{n+i}}\!\G.$$
	Therefore
	\begin{equation*}  \dim(H_1[n+i]\cap H_2[n+i])\leq \dim((H_1\cap H_2)[m+i])+(n-m)\dim(\G) \end{equation*}
	and so
	\begin{align*}
	\dim((H_1\cap & H_2)[m+i])  \geq \dim(H_1[n+i]\cap H_2[n+i])-(n-m)\dim(\G) \\
	& \geq \dim(H_1[n+i])+\dim(H_2[n+i])-\dim(G[n+i])-(n-m)\dim(\G).
	\end{align*}
	Now using Theorem \ref{theo: existence sdim} and comparing the coefficients of $i$ yields identity (\ref{eqn: sdimintersect}).
	
	It remains to prove (\ref{eqn: orderintersect}). Obviously this is true if $\ord(G)=\infty$. So we can assume $\ord(G)<\infty$, i.e., $\sdim(G)=0$. Note that (\ref{eqn: proof}) implies that the intersection of \mbox{$\I(H_1[n+i]\cap H_2[n+i])$} with $k[\G[m+i]]$ equals $\I((H_1\cap H_2)[m+i])$. This means that the morphism of algebraic groups
	$$\f\colon H_1[n+i]\cap H_2[n+i]\to (H_1\cap H_2)[m+i]$$
	induced from the projection $\G\times{{}^{\s}\!\G}\times\cdots\times {}^{\s^{n+i}}\!\G\to \G\times{{}^{\s}\!\G}\times\cdots\times {}^{\s^{m+i}}\!\G$ is
	a quotient map. Since $\sdim(H_1)=0$, the kernel of the projections $H_1[n+i]\to H_1[m+i]$ is finite for $i\gg 0$. Therefore, also $\f$ has finite kernel and it follows that for $i\gg 0$
	\begin{align*}\ord(H_1\cap H_2)&=\dim((H_1\cap H_2)[m+i]) =\dim(H_1[n+i]\cap H_2[n+i]) \\
	& \geq \dim(H_1[n+i])+\dim(H_2[n+i])-\dim(G[n+i])\\
	& =\ord(H_1)+\ord(H_2)-\ord(G).
	\end{align*}
\end{proof}

\section{The limit degree} \label{sec: limit degree}

The limit degree is a classical and important numerical invariant of extensions of difference fields. See \cite[Section 4.3]{Levin:difference}. Based on the results of Section \ref{sec: Zariski closures of difference algebraic groups and the growth group}, we introduce the limit degree of a difference algebraic group. A connection between the two notions is that the limit degree of a $\s$-Picard-Vessiot extension (in the sense of \cite{DiVizioHardouinWibmer:DifferenceGaloisTheoryOfLinearDifferentialEquations}) agrees with the limit degree of the corresponding $\s$-Galois group. We also note that, in our restricted affine setting, difference algebraic groups of limit degree one correspond to \emph{algebraic $\s$-groups} (\cite{KowalskiPillay:OnAlgebraicSigmaGroups}).
%

By the \emph{size} $|\G|$ of an algebraic group $\G$ we mean the dimension of $k[\G]$ as a $k$-vector space. So the size is either a non-negative integer or $\infty$ (and may exceed $|\G(k)|$). In the sequel we will employ the usual rules for calculating with the symbol $\infty$.
E.g., if $\G_1\xrightarrow{\f_1}\G_2\xrightarrow{\f_2} \G_3$
 are quotient maps of algebraic groups, then
\begin{equation}\label{eqn: multiplicativity of size}
|\ker(\f_2\circ\f_1)|=|\ker(\f_2)|\cdot |\ker(\f_1)|.
\end{equation}

In Section \ref{sec: Zariski closures of difference algebraic groups and the growth group} we have seen that the dimension of the growth group of a $\s$-closed subgroup $G$ of an algebraic group $\G$ only depends on $G$. The following proposition shows that the same is true for the size of the growth group.

\begin{prop} \label{prop: existence of limit degree}
	Let $G$ be a $\s$-algebraic group and $G\hookrightarrow\G$ a $\s$-closed embedding. Then the size of the growth group of $G$ with respect to the embedding $G\hookrightarrow\G$ does not depend on the choice of $\G$ and the $\s$-closed embedding.
\end{prop}
\begin{proof}
	For $i\geq 0$ let $G[i]$ denote the $i$-th order Zariski closure of $G$ in $\G$ and $\G_i$ the kernel of the projection $\pi_i\colon G[i]\to G[i-1]$. By Proposition \ref{prop: existence of growth group} the integer $d=|\G_i|$ does not depend on $i$ for $i\gg0$. Assume that the finite set $A$ generates $k[G[0]]\subseteq k\{G\}$ as a $k$-algebra. Then $A$ also $\s$-generates $k\{G\}$ as a $k$-$\s$-algebra.
	
	Let $\G'$ be another algebraic group and $G\hookrightarrow\G'$ a $\s$-closed embedding. Let $G[i]'$ denote the $i$-th order Zariski closure of $G$ in $\G'$ and let $d'$ and $A'$ be as above. We have to show that $d=d'$.
	
	Since $A$, as well as $A'$, $\s$-generate $k\{G\}$, there exists an integer $m\geq 1$ such that
	$A'\in k[A,\ldots,\s^m(A)]$ and $A\in k[A',\ldots,\s^m(A')]$. Then, for $i\geq 0$, we have $k[A',\ldots,\s^i(A')]\subseteq k[A,\ldots,\s^{m+i}(A)]$ and $k[A,\ldots,\s^i(A)]\subseteq k[A',\ldots,\s^{m+i}(A')]$. So for $j\geq m$:
	$$k[A,\ldots,\s^i(A)]\subseteq k[A',\ldots,\s^{m+i}(A')]\subseteq k[A',\ldots,\s^{j+i}(A')]\subseteq k[A,\ldots,\s^{m+j+i}(A)].$$
	These inclusions of Hopf algebras correspond to quotient maps of algebraic groups
	$$G[m+j+i]\to G[j+i]'\to G[m+i]'\to G[i].$$
	We have $|\ker(G[m+j+i]\to G[i])|\geq |\ker(G[j+i]'\to G[m+i]')|$ by (\ref{eqn: multiplicativity of size}).
	But by (\ref{eqn: multiplicativity of size}) and Proposition \ref{prop: existence of growth group} we also have $$|\ker(G[m+j+i]\to G[i])|=d^{m+j} \quad \text{ and } \quad |\ker(G[j+i]'\to G[m+i]')|=d'^{j-m}$$ for $i\gg 0$. Consequently, $d^{m+j}\geq d'^{j-m}$. Letting $j$ tend to infinity, we find $d\geq d'$. By symmetry, $d=d'$.
\end{proof}

\begin{defi}
	Let $G$ be a $\s$-algebraic group. Choose an algebraic group $\G$ and a \mbox{$\s$-closed} embedding $G\hookrightarrow\G$. The size of the growth group of $G$ with respect to the $\s$-closed embedding $G\hookrightarrow\G$ is called the \emph{limit degree}\index{limit degree} of $G$ and is denoted by
	$\ld(G).$
	By Proposition \ref{prop: existence of limit degree} the limit degree of $G$ does not depend on the choice of $\G$ and the $\s$-closed embedding $G\hookrightarrow\G$.
\end{defi}
The expression ``limit degree'' is motivated by the fact that
$\ld(G)=\lim_{i\to\infty}\deg(\pi_i),$
where $\deg(\pi_i)$ denotes the degree of the projection $\pi_i\colon G[i]\to G[i-1]$. The naming is also motivated by Lemma \ref{lemma: correspondence of ld} below. 
We note that, by Corollary \ref{cor: sdim equals dim of growth group}, the limit degree of a $\s$-algebraic group is finite if and only if it has $\s$-dimension zero.

\begin{ex}
	Let $0\leq \alpha_1<\cdots<\alpha_n$ and $1\leq\beta_1,\ldots,\beta_n$ be integers and $G\leq\Gm$ the $\s$-closed subgroup of the multiplicative group $\Gm$ given by $$G(R)=\{g\in R^\times|\ \s^{\alpha_1}(g)^{\beta_1}\cdots\s^{\alpha_n}(g)^{\beta_n}=1\}$$
	for every $k$-$\s$-algebra $R$. Then $\ld(G)=\beta_n$ by Example \ref{ex: ssubgroup of Gm}.
\end{ex}

\begin{ex}
	By Example \ref{ex: growth group of linear} the limit degree of a $\s$-closed subgroup of $\Ga$ defined by a linear difference equation is one.
\end{ex}

\begin{ex} \label{ex: ld of algebraic group}
	Let $\G$ be an algebraic group. Then $\ld([\s]_k\G)=|\G|$ by Example \ref{ex: growth group of algebraic group}.
	
\end{ex}

We will next show that our definition of the limit degree corresponds to the classical notion (whenever both notions make sense). To this end, let us recall the definition of the limit degree of an extension $K/k$ of $\s$-fields. Assume that there exists a finite set $A\subseteq K$ such that $a,\s(A),\ldots$ generates $K$ as a field extension of $k$, then the limit degree $\ld(K/k)$ is the limit $\lim_{i\to\infty} d_i$, where $d_i$ is the degree of the field extension $k(A,\ldots,\s^{i}(A))/k(A,\ldots,\s^{i-1}(A)).$ The limit exists and does not depend on the choice of $A$ (\cite[Section 4.3]{Levin:difference}).

\begin{lemma} \label{lemma: correspondence of ld}
	Let $G$ be a $\s$-algebraic group such that $k\{G\}$ is an integral domain and $\s\colon k\{G\}\to k\{G\}$ is injective. Then the limit degree of $G$ equals the limit degree of the field of fractions of $k\{G\}$ over $k$.
\end{lemma}
\begin{proof}
	Let $\G$ be an algebraic group containing $G$ as a $\s$-closed subgroup. For $i\geq 0$ let $G[i]$ denote the $i$-th order Zariski closure of $G$ in $\G$ and let $A\subseteq k[G[0]]$ be a finite set that generates $k[G[0]]\subseteq k\{G\}$ as a $k$-algebra. Let $K$ denote the field of fractions of $k\{G\}$. Then $A,\s(A),\ldots$ generates $K$ as a field extension of $k$ and $k(A,\s(A),\ldots,\s^i(A))$ is the field of fractions of $k[G[i]]$. Therefore the degree of the field extension $$k(A,\ldots,\s^{i}(A))/k(A,\ldots,\s^{i-1}(A))$$ equals the degree of the projection $G[i]\to G[i-1]$.
\end{proof}

%


It is well-known that a finitely $\s$-generated extension of $\s$-fields has limit degree one if and only if the extension is finitely generated as a field extension (\cite[Prop. 4.3.12]{Levin:difference}). A similar statement is true for difference Hopf algebras:

\begin{lemma} \label{lemma: ld equal 1}
	Let $G$ be a $\s$-algebraic group. Then $\ld(G)=1$ if and only if $k\{G\}$ is finitely generated as a $k$-algebra.
\end{lemma}
\begin{proof}
	Fix a $\s$-closed embedding $G\hookrightarrow\G$ and let $G[i]$ denote the $i$-th order Zariski closure of $G$ in $\G$. So $k\{G\}=\cup_{i\geq 0}k[G[i]]$.
	If $\ld(G)=1$, there exists an integer $m\geq 0$ such that the quotients maps $\pi_i\colon G[i]\to G[i-1]$ have trivial kernel for $i> m$. So $\pi_i$ is an isomorphism and therefore $k[G[i]]=k[G[i-1]]$. Consequently, $k\{G\}=k[G[m]]$ is finitely generated as a $k$-algebra.
	
	Conversely, if $k\{G\}$ is a finitely generated $k$-algebra, we can consider the algebraic group $\G$ associated with the Hopf algebra $k\{G\}$, i.e., $k[\G]=k\{G\}$ as $k$-algebras.
	By Lemma \ref{lemma: Hopf adjunction} the morphism $k\{\G\}\to k\{G\}$ induced from the identity map $k[\G]\to k\{G\}$ is a morphism of \ks-Hopf algebras.	
	 With respect to the corresponding $\s$-closed embedding $G\hookrightarrow\G$ we have $k\{G\}=k[G[0]]=k[G[1]]=\ldots$ and therefore the associated growth group is trivial. Thus $\ld(G)=1$.
\end{proof}

\begin{ex}
	It is clear from Lemma \ref{lemma: ld equal 1} that the $\s$-algebraic groups in Examples~\ref{ex: general constr}, \ref{ex: Frobenius algebraic group} and \ref{ex: split finite} have limit degree one.
\end{ex}

We conclude this section by explaining the connection between the category of algebraic $\s$-groups introduced in \cite{KowalskiPillay:OnAlgebraicSigmaGroups} and affine difference algebraic groups of limit degree one.
Let us begin by recalling the definition of algebraic $\s$-groups from \cite{KowalskiPillay:OnAlgebraicSigmaGroups}. Let $k$ be a $\s$-field. 
An \emph{algebraic $\s$-variety\footnote{The assumptions in \cite{KowalskiPillay:OnAlgebraicSigmaGroups} on $k$ and $\X$ are slightly more restrictive but this is irrelevant for our discussion here.}} over $k$ is a scheme $\X$ of finite type over $k$ together with a morphism $\wtilde{\s}\colon\X\to{\hs\X}$ of schemes over $k$. Here, as in Section \ref{subsec: Zariski closures}, $\hs \X$ denotes the scheme over $k$ obtained from $\X$ by base change via $\s\colon k\to k$. 
 A morphism between algebraic $\s$-varieties is a morphism $\f\colon\X\to\Y$ of schemes over $k$ such that
$$
\xymatrix{
	\X \ar^-{\wtilde{\s}}[r] \ar_\f[d] & \hs \X \ar^-{\hs\f}[d]\\
	\Y \ar^-{\wtilde{\s}}[r] & \hs \Y
}
$$
commutes. An \emph{algebraic $\s$-group} is a group object in the category of algebraic $\s$-varieties. Equivalently, an algebraic $\s$-group is a (not necessarily affine) algebraic group $\G$ over $k$ together with a morphism $\wtilde{\s}\colon\G\to{\hs\G}$ of algebraic groups.
The category of \emph{affine algebraic $\s$-groups} consists of the algebraic $\s$-groups whose underlying scheme is affine; the morphisms are morphism of algebraic $\s$-varieties that are also morphisms of algebraic groups.

\begin{prop} \label{prop: algebraic sgroups}
	The category of affine algebraic $\s$-groups is equivalent to the category of \mbox{$\s$-algebraic} groups of limit degree one.
\end{prop}
\begin{proof}
	Let $R$ be a $k$-algebra. To define a $k$-$\s$-algebra structure on $R$ is equivalent to defining a morphism of $k$-algebras $\overline{\s}\colon {\hs R}\to R$: Given $\s\colon R\to R$, we can define $$\overline{\s}\colon {\hs R}\to R,\ r\otimes\lambda\mapsto \s(r)\lambda.$$ Conversely, given $\overline{\s}\colon {\hs R}\to R$, we can define $\s\colon R\xrightarrow{\id\otimes 1}R\otimes_k k={\hs R}\xrightarrow{\overline{\s}} R$.
	Moreover, if $R$ and $S$ are $k$-$\s$-algebras and $\psi\colon R\to S$ is a morphism of $k$-algebras, then $\psi$ is a morphism of $k$-$\s$-algebras if and only if
	$$
	\xymatrix{
		\hs R \ar^-{\overline{\s}}[r] \ar_-{\hs\psi}[d] & R \ar^\psi[d]\\
		\hs S \ar^-{\overline{\s}}[r] & S
	}
	$$
	commutes.
	
	Dualizing the definition, an affine algebraic $\s$-group $\G$ corresponds to a finitely generated Hopf algebra $k[\G]$ together with a morphism of Hopf algebras $\overline{\s}=({\wtilde{\s}})^*\colon{\hs (k[\G])}\to k[\G]$. By the remark from the beginning of the proof, the statement that $\overline{\s}$ is a morphism of Hopf algebras corresponds to the statement that $k[\G]$ is a $k$-$\s$-Hopf algebra. Thus the category of affine algebraic $\s$-groups is anti-equivalent to the category of $k$-$\s$-Hopf algebras that are finitely generated as $k$-algebras. Now the claim follows from Remark \ref{rem: equivalence sgroups sHopfalgebras} and Lemma~\ref{lemma: ld equal 1}.
\end{proof}

\section{Connected components of difference algebraic groups} \label{section: components}

In this section we study the connected components of $\spec(k\{G\})$ for a difference algebraic group $G$. Our main result is that there are only finitely many invariant connected components. This result can be reinterpreted as a positive answer (in a special case) to a question raised by E. Hrushovski.

\subsection{Two equivalent conjectures}
%

In \cite[Section 4.6]{Hrushovski:elementarytheoryoffrobenius} E. Hrushovski raised the question whether or not it is possible to strengthen the classical Ritt-Raudenbusch basis theorem (\cite[Theorem 2.5.11]{Levin:difference}). For clarity, let us state the question as a conjecture:

\begin{conj} \label{conj: Hrushovski1}
	Let $k$ be a $\s$-field and $R$ a finitely $\s$-generated $k$-$\s$-algebra. Then every ascending chain of radical, well-mixed $\s$-ideals in $R$ is finite.
\end{conj}
Recall that a $\s$-ideal $\ida$ is \emph{well-mixed}, if $ab\in\ida$ implies $a\s(b)\in\ida$. E. Hrushovski proved Conjecture \ref{conj: Hrushovski1} under certain additional assumptions on $R$ (\cite[Lemma 4.35]{Hrushovski:elementarytheoryoffrobenius}). Moreover, J. Wang (\cite{Wang:FiniteBasisForRadicalWellmixedDifferenceIdeals} \cite{Wang:MonomialDifferenceIdeals}) proved the conjecture (appropriately reformulated) for monomial and binomial difference ideals.
In \cite{Levin:OnTheAscendingChainCondition} A. Levin showed that the conjecture fails if the assumption that the $\s$-ideals are radical is dropped.

A prime $\s$-ideal $\p$ of a $\s$-ring $R$ is a $\s$-ideal that is a prime ideal. It is a \emph{minimal prime $\s$-ideal} if for every prime $\s$-ideal $\q$ of $R$ with $\q\subseteq \p$ we have $\q=\p$. We will show below that Conjecture \ref{conj: Hrushovski1} is equivalent to the following conjecture.

\begin{conj} \label{conj: Hrushovski2}
	Let $k$ be a $\s$-field and $R$ a finitely $\s$-generated $k$-$\s$-algebra. Then the set of minimal prime $\s$-ideals of $R$ is finite.
\end{conj}
The main result of this section is a special case of Conjecture \ref{conj: Hrushovski2}:
\begin{theo} \label{theo: Hrushovskiconj}
	Conjecture \ref{conj: Hrushovski2} holds under the additional assumption that $R$ can be equipped with the structure of a $k$-$\s$-Hopf algebra.
\end{theo}

The proof of Theorem \ref{theo: Hrushovskiconj} is given at the end of this section.
We will first show that the two conjectures are equivalent. To this end
we need a basic topological lemma. Recall that a topological space $X$ is \emph{Noetherian} if every descending chain of closed subsets of $X$ is finite.

\begin{lemma} \label{lemma: topological}
	A topological space is Noetherian if and only if
	\begin{enumerate}
		\item every descending chain of irreducible closed subsets is finite and
		\item every closed subset is a finite union of irreducible closed subsets.
	\end{enumerate}
\end{lemma}
Assume that $X$ is a Noetherian topological space. Then clearly (i) is satisfied and it is also well-known that (ii) holds for $X$ (\cite[\href{https://stacks.math.columbia.edu/tag/0050}{Lemma 5.9.2, Tag 0050}]{stacks-project}).

The proof of the reverse direction is a variation of the proof of K\"{o}nig's Lemma: Assume that $X$ is a topological space satisfying (i) and (ii). Let 
\begin{equation} \label{eq: descending sequence}
X_1\supseteq X_2\supseteq \ldots
\end{equation} be a descending chain of closed subsets of $X$. Using (ii), for every $i\geq 1$, we may write $X_i=Y_{i,1}\cup\ldots\cup Y_{i,n_i}$ as an irredundant union of irreducible closed subsets $Y_{i,j}$ of $X$. 
We define a directed graph as follows: The set of vertices is $\{(i,j)\ | \ 1\leq i \text{ and } 1\leq j\leq n_i\}$ and there is an arrow from $(i,j)$ to $(i',j')$ if $i'=i-1$ and $Y_{i,j}\subseteq Y_{i',j'}$.

Let $i\geq 2$ and $1\leq j\leq n_i$. Since $X_{i}\subseteq X_{i-1}$ and $Y_{i,j}$ is irreducible, there exists a $j'$ with $1\leq j'\leq n_{i-1}$ and $Y_{i,j}\subseteq Y_{i-1,j'}$. So we see that for every vertex on the $i$-level (i.e., for ever vertex of the form $(i,j)$) there exists an arrow to a vertex on the $i-1$ level. In particular, for every vertex there exists a path to the $1$-level.

If $S_1=\{Y_{i,j}\ | \ 1\leq i \text{ and } 1\leq j\leq n_i\}$ is finite, then also (\ref{eq: descending sequence}) is finite. So we suppose that $S_1$ is infinite. Since $S_1$ is infinite, there must exists a $j_1$ with $1\leq j_1\leq n_1$ such that $$S_2=\left\{Y_{i,j} \ | \ i\geq 2,\ 1\leq j\leq n_i, \text{ there exists a path from $(i,j)$ to $(1,j_1)$}\right\}$$ is infinite. Set $Z_1=Y_{1,j_1}$. Since $S_2$ is infinite, there exists a $j_2$ with $1\leq j_2\leq n_2$ such that $Y_{2,j_2}\in S_2$ and 
$$S_3=\{Y_{i,j} \ |\ i\geq 3,\ 1\leq j\leq n_i, \text{ there exists a path from $(i,j)$ to $(2,j_2)$}\}$$
is infinite. Set $Z_2=Y_{2,j_2}$. Then $Z_1\supseteq Z_2$. Continuing like this, we obtain a descending chain $Z_1\supseteq Z_2\supseteq Z_3\supseteq\ldots$ of irreducible closed subsets of $X$ and a descending chain $S_1\supseteq S_2\supseteq S_3\supseteq\ldots$ of infinite sets.

We claim that $Z_i=Z_{i+1}$ implies $S_i=S_{i+1}$. But if $Z_i=Z_{i+1}$, then there is exactly one arrow from level $i+1$ to level $i$ that points to $(i,j_i)$ (because the representations $X_i=Y_{i,1}\cup\ldots\cup Y_{i,n_i}$ are irredundant). So $S_i=S_{i+1}$.

By (i) we have $Z_i=Z_{i+1}$ for $i\gg 1$. Therefore $S_i=S_{i+1}$ for $i\gg 1$. But then $S_i$ must consist of a single element for $i\gg 1$. This contradicts the fact that $S_i$ is infinite.

\begin{prop} \label{prop: conjectures equivalent}
	Let $R$ be a finitely $\s$-generated \ks-algebra. Then the following two statements are equivalent:
	\begin{enumerate}
		\item Every ascending chain of radical well-mixed $\s$-ideals in $R$ is finite.
		\item Every radical, well-mixed $\s$-ideal of $R$ is the finite intersection of prime $\s$-ideals.
	\end{enumerate}
\end{prop}
\begin{proof}
	We consider the set $X$ of all prime $\s$-ideals of $R$ as a topological space with respect to the topology induced from the Zariski topology on $\spec(R)$. Since every radical well-mixed $\s$-ideal $\ida$ of $R$ is the intersection of the prime $\s$-ideals of $R$ that contain $\ida$ (\cite[Lemma 2.10]{Hrushovski:elementarytheoryoffrobenius} or \cite[Prop. 1.2.28]{Wibmer:LectureNotes}), it follows that the map $\ida\mapsto V(\ida)=\{\p \ |\ \ida\subseteq \p\}$ from the set of radical well-mixed $\s$-ideals of $R$ to the set of closed subsets of $X$ is bijective. Moreover, for a radical, well-mixed $\s$-ideal $\ida$, the closed subset $V(\ida)$ is irreducible if and only if $\ida$ is prime.

	An ascending chain of radical well-mixed $\s$-ideals in $R$ corresponds to a descending chain of closed subset of $X$. Thus, condition (i) translates to $X$ being a Noetherian topological space. On the other hand, condition (ii) translates to: Every closed subset of $X$ is a finite union of irreducible closed subsets. 
	
	It is known that every ascending chain of prime $\s$-ideals of $R$ is finite (\cite[Lemma 4.34]{Hrushovski:elementarytheoryoffrobenius} or \cite[Cor. 5.1.6]{Wibmer:LectureNotes}). In other words, every descending chain of irreducible closed subsets of $X$ is finite. Thus the claim follows from Lemma \ref{lemma: topological}.
	\end{proof}
From Proposition \ref{prop: conjectures equivalent} we immediately obtain: 
\begin{cor}
	Conjectures \ref{conj: Hrushovski1} and \ref{conj: Hrushovski2} are equivalent.
\end{cor}

\subsection{Connected components}

For the proof of Theorem \ref{theo: Hrushovskiconj} we need to introduce the group of connected components $\pi_0(G)$ of a $\s$-algebraic group $G$. We do not aim for an exhaustive study of $\pi_0(G)$ and the closely related notion of the identity component $G^o$ here. (The interested reader is referred to \cite[Section 4]{Wibmer:Habil}.) We rather work towards a proof of Theorem \ref{theo: Hrushovskiconj} as directly as possible.

An \emph{irreducible (or connected) component} of a topological space is a maximal irreducible (or connected) subset. An irreducible (or connected) component is closed. Every topological space is the disjoint union of its connected components.

A \emph{connected (or irreducible) component} of a $\s$-algebraic group $G$ is a connected (or irreducible) component of $\spec(k\{G\})$. If $R$ is a ring and $\ida\subseteq R$, we denote by $\VV(\ida)$ the closed subset of $\spec(R)$ defined by $\ida$. The following lemma is really about affine group scheme, rather than difference algebraic groups. But for lack of a suitable reference we include the proof. 

\begin{lemma} \label{lemma:connectedcomponentsareirreducible}
	Let $G$ be a $\s$-algebraic group. The connected components and the irreducible components of $G$ coincide. Moreover, if $\p$ is a prime ideal of $k\{G\}$, then the connected component of $G$ containing $\p$ equals $\VV(\ida)$, where $\ida$ is the ideal of $k\{G\}$ generated by all idempotent elements of $k\{G\}$ contained in $\p$.
\end{lemma}
\begin{proof}
	Let us fix a $\s$-closed embedding $G\hookrightarrow\G$ of $G$ into some algebraic group $\G$ and for $i\geq 0$ let $G[i]$ denote the $i$-th order Zariski-closure of $G$ in $\G$. Let $C\subseteq\spec(k\{G\})$ denote a connected component of $G$. Then $C=\VV(\ida)$ for a unique ideal $\ida$ of $k\{G\}$ generated by idempotent elements. (See \cite[Tag 00EB]{stacks-project}.) For every $i\geq 0$, the closure of the image of $C$ under the projection $\spec(k\{G\})\to G[i]$ is connected and equal to $\VV(\ida\cap k[G[i]])\subseteq G[i]$. So $\VV(\ida\cap k[G[i]])\subseteq G[i]$ is contained in a connected component of $G[i]$. Assume that $G[i]$ has $n_i$ connected components. Then $$k[G[i]]=e_{i,1}k[G[i]]\oplus\cdots\oplus e_{i,n_i}k[G[i]]$$ for some primitive idempotent elements $e_{i,1},\ldots,e_{i,n_i}\in k[G[i]]$ and $\VV(\ida\cap k[G[i]])\subseteq \VV(\idb_i)$ where $\idb_i=(e_{i,1},\ldots e_{i,j_i-1}, e_{i,j_i+1},\ldots,e_{i,n_i})\subseteq k[G[i]]$ for a unique $j_i\in \{1,\ldots,n_i\}$.
	We have $\idb_{i+1}\cap k[G[i]]=\idb_i$ for $i\geq 0$ and $\idb:=\cup_{i\geq 0} \idb_i$ is an ideal of $k\{G\}$.
	From $\VV(\ida\cap k[G[i]])\subseteq \VV(\idb_i)$ it follows that $\idb_i$ is contained in the radical of $\ida\cap k[G[i]]$. Since the $e_{i,j}$'s are idempotent this shows that $\idb_i\subseteq\ida$. So $\idb\subseteq\ida$.
	
	Since $k\{G\}/\idb$ may be interpreted as the directed union of the algebras $k[G[i]]/\idb_i\simeq e_{i,j_i}k[G[i]]$ which have a prime nilradical, it is clear that $\VV(\idb)$ is irreducible (and a fortiori connected). As $C=\VV(\ida)\subseteq \VV(\idb)$ it follows from the maximality of $C$ that $\VV(\ida)=\VV(\idb)$. By the uniqueness of $\ida$ we have $\ida=\idb$.
	
	We have thus shown that every connected component of $G$ is irreducible. So the connected and the irreducible components of $G$ coincide.
	The claimed form of the connected component of a prime ideal of $k\{G\}$ follows from the above arguments.
\end{proof}

By Lemma \ref{lemma:connectedcomponentsareirreducible}, the connected components of a $\s$-algebraic group $G$ are in bijection with the minimal prime ideals of $k\{G\}$. Contrary to the case of algebraic groups or differential algebraic groups, a difference algebraic group can have infinitely many components. Moreover, the components need not be open. This is illustrated in the following example. 

\begin{ex} \label{ex: infinitely many components}
	Let $G\leq\Gm$ be the $\s$-algebraic group given by
	$$G(R)=\{g\in R^\times|\ g^2=1\}\leq\Gm(R)$$
	for any $k$-$\s$-algebra $R$. We have $$k\{G\}=k\{y\}/[y^2-1]=k[y,\s(y),\ldots]/(y^2-1,\s(y)^2-1,\ldots).$$ Let us assume that the characteristic of $k$ is not equal to $2$. Then the prime ideals of $k\{G\}$ are in bijection with the set of all sequences $(a_i)_{i\in\mathbb{N}}$ such that $a_i\in\{1,-1\}$.
	Every prime ideal of $k\{G\}$ is its own component. In particular, $G$ has infinitely many components. The open subsets of $\spec(k\{G\})$ are all infinite, thus the components are not open.
\end{ex}
For a $\s$-algebraic group $G$ we define $$\Sigma\colon\spec(k\{G\})\to\spec(k\{G\}),\ \p\mapsto\s^{-1}(\p),$$ where $\s\colon k\{G\}\to k\{G\}$. A connected component $C$ of $G$ is \emph{invariant} if $\Sigma(C)\subseteq C$.
%

\begin{ex}
	The $\s$-algebraic group from Example \ref{ex: infinitely many components} has two invariant connected components, namely the prime ideals corresponding to the sequences $(1,1,\ldots)$ and $(-1,-1,\ldots)$.
\end{ex}

\begin{lemma} \label{lemma: scomponent}
	Let $G$ be a $\s$-algebraic group and $C\subseteq\spec(k\{G\})$ a connected component of $G$. Then $C$ is invariant if and only if $C$ contains a prime $\s$-ideal.
\end{lemma}
\begin{proof}
	Assume first that $C$ is invariant. By Lemma \ref{lemma:connectedcomponentsareirreducible} the connected component $C$ is of the form $C=\VV(\p)$, for some prime ideal $\p$ of $k\{G\}$. Since $C$ is invariant, $\s^{-1}(\p)\in\VV(\p)$, i.e., $\p\subseteq \s^{-1}(\p)$. Thus $\s(\p)\subseteq\p$ and $\p\in C$ is a prime $\s$-ideal.
	
	Conversely, assume that $\p\in C$ is a prime $\s$-ideal. It follows from Lemma \ref{lemma:connectedcomponentsareirreducible} that $C=\VV(E)$, where $E$ is the set of idempotent elements contained in $\p$. Since $\sigma(\p)\subseteq \p$ we have $\s(E)\subseteq E$ and it follows that $\Sigma(C)\subseteq C$.
\end{proof}

\begin{cor} \label{cor: minsprime is minimal}
	Let $k\{G\}$ be a $k$-$\s$-Hopf algebra that is finitely $\s$-generated over $k$. Then a minimal prime $\s$-ideal of $k\{G\}$ is a minimal prime ideal of $k\{G\}$.
\end{cor}
\begin{proof}
	Let $\q\subseteq k\{G\}$ be a minimal prime $\s$-ideal and let $C\subseteq\spec(k\{G\})$ be the connected component that contains $\q$. Then $C=\VV(\p)$ for a minimal prime ideal $\p$ of $k\{G\}$. Since $\q\in C$ it follows from Lemma \ref{lemma: scomponent} that $C$ is invariant and so $\p$ is a $\s$-ideal. Therefore $\q=\p$ by the minimality of $\q$. 
\end{proof}


We will next introduce the group of connected components $\pi_0(G)$ of a $\s$-algebraic group $G$. This is a rather straight forward adaption of the standard construction for algebraic groups. See \cite[Sections 5 and 6]{Waterhouse:IntroductiontoAffineGroupSchemes} or \cite[Section 2.g]{Milne:AlgebraicGroupsTheTheoryOfGroupSchemesOfFiniteTypeOverAField}.
We begin by defining appropriate analogs of \'{e}tale algebras and \'{e}tale algebraic groups.
%

\begin{defi}
	A finitely $\s$-generated $k$-$\s$-algebra $R$ is \emph{$\s$-\'{e}tale} (over $k$) if $R$ is integral over $k$ and a separable $k$-algebra. A $\s$-algebraic group $G$ is \emph{$\s$-\'{e}tale} if $k\{G\}$ is $\s$-\'{e}tale over $k$.
\end{defi}

Thus, a finitely $\s$-generated $k$-$\s$-algebra $R$ is $\s$-\'{e}tale if and only if for every $r\in R$ there exists a separable polynomial $f$ over $k$ with $f(r)=0$. Related notions of \'{e}taleness in difference algebra occur in \cite{Tomasic:ATwistedTheoremOfChebotarev}, \cite{Tomasic:TwistedGaloisStratification} and \cite{TomasicWibmer:Babbit}.


The $\s$-algebraic group from Example \ref{ex: split finite} is $\s$-\'{e}tale. Moreover, if $\G$ is an \'{e}tale algebraic group over $k$, then $[\s]_k\G$ is a $\s$-\'{e}tale $\s$-algebraic group. $\s$-\'{E}tale $\s$-algebraic groups have a rather rigid structure which we will not discuss here. The interested reader is referred to Section 6 of \cite{Wibmer:Habil}.
%

Recall (\cite[Chapter V, \S 6]{Bourbaki:Algebra2}) that a $k$-algebra $A$ is \'{e}tale if $A\otimes_k\overline{k}$ is a finite direct product of copies of $\overline{k}$, where $\overline{k}$ denotes the algebraic closure of $k$. For a $k$-algebra $A$ we let
$\pi_0(A)$
denote the union of all \'{e}tale $k$-subalgebras of $A$. That is, $\pi_0(A)$ consists of all elements of $A$ that annul a separable polynomial over $k$. Then $\pi_0(A)$ is a $k$-subalgebra of $A$. (Cf. Section~6.7 in \cite{Waterhouse:IntroductiontoAffineGroupSchemes}.) 


\begin{prop} \label{prop: pi0}
	Let $G$ be a $\s$-algebraic group. There exists a $\s$-\'{e}tale $\s$-algebraic group $\pi_0(G)$ and a morphism $G\to \pi_0(G)$ of $\s$-algebraic groups satisfying the following universal property: If $G\to H$ is a morphism of $\s$-algebraic groups with $H$ $\s$-\'{e}tale, then there exists a unique morphism $\pi_0(G)\to H$ such that
	\[\xymatrix{
		G \ar[rr] \ar[rd] & & \pi_0(G) \ar@{..>}[ld] \\
		& H &
	}
	\]
	commutes.
\end{prop}
\begin{proof}
	The $k$-subalgebra $\pi_0(k\{G\})$ of $k\{G\}$ is a Hopf subalgebra. (Cf. Section 6.7 in \cite{Waterhouse:IntroductiontoAffineGroupSchemes}.)
	Let $a\in \pi_0(k\{G\})$ and let $f$ be a separable polynomial over $k$ with $f(a)=0$. Let ${\hs f}$ denote the polynomial obtained from $f$ by applying $\s$ to the coefficients. Since $1\in (f,f')$, it follows that also $1\in( {\hs f}, ({\hs f})')$. Thus ${\hs f}$ is separable. Since ${\hs f}(\s(a))=0$, we see that $\s(a)\in\pi_0(k\{G\})$. This shows that $\pi_0(k\{G\})$ is a $k$-$\s$-Hopf subalgebra of $k\{G\}$. Note that $\pi_0(k\{G\})$ is finitely $\s$-generated over $k$ by Theorem \ref{theo: finiteness2}. Let $\pi_0(G)$ denote the $\s$-\'{e}tale $\s$-algebraic group represented by $\pi_0(k\{G\})$ and let $G\to \pi_0(G)$ be the morphism corresponding to the inclusion $\pi_0(k\{G\})\subseteq k\{G\}$ (Remark \ref{rem: equivalence sgroups sHopfalgebras}).
	
	If $G\to H$ is a morphism of $\s$-algebraic groups with $H$ $\s$-\'{e}tale, then the image of the dual map $k\{H\}\to k\{G\}$ consists of elements that annul a separable polynomial. Thus the image lies in $\pi_0(k\{G\})$ and $k\{H\}\to k\{G\}$ factors uniquely through \mbox{$\pi_0(k\{G\})\hookrightarrow k\{G\}$}.
\end{proof}
Of course $\pi_0(G)$ is unique up to a unique isomorphism. 

\begin{defi}
	Let $G$ be a $\s$-algebraic group. The $\s$-\'{e}tale $\s$-algebraic group $\pi_0(G)$ defined by the universal property of Proposition \ref{prop: pi0} is the \emph{group of connected components} of $G$. The kernel $G^o$ of $G\to\pi_0(G)$ is the \emph{identity component} of $G$.
\end{defi}
The next lemma reduces the proof of Theorem \ref{theo: Hrushovskiconj} to $\s$-\'{e}tale \ks-algebras.

\begin{lemma} \label{lemma: correspondence components}
	Let $G$ be a $\s$-algebraic group. There is a one-to-one correspondence between the connected components of $G$ and the connected components of $\pi_0(G)$. Under this bijection invariant connected components correspond to invariant connected components. Moreover, every connected component of $\pi_0(G)$ is a single point.
\end{lemma}
\begin{proof}
	Every prime ideal of $k\{\pi_0(G)\}$ is maximal and hence also minimal. This shows that the connected components of $\pi_0(G)$ are points. We identify $k\{\pi_0(G)\}$ with its image $\pi_0(k\{G\})$ in $k\{G\}$. We claim that $\p\mapsto \p\cap k\{\pi_0(G)\}$ is a bijection between the minimal prime ideals of $k\{G\}$ and the (minimal) prime ideals of $k\{\pi_0(G)\}$. Every (minimal) prime ideal of $k\{\pi_0(G)\}$ is of the form $\p\cap k\{\pi_0(G)\}$ for some minimal prime ideal of $k\{G\}$ (\cite[Proposition 16, Chapter II, \S 2.6]{Bourbaki:commutativealgebra}). On the other hand, if $\p$ is a minimal prime ideal of $k\{G\}$, we see, using Lemma \ref{lemma:connectedcomponentsareirreducible}, that
	$\VV(\p)=\VV(\ida)$, where $\ida$ is the ideal of $k\{G\}$ generated by all idempotent elements of $k\{G\}$ contained in $\p$. Thus $\p=\sqrt{\ida}$. Since all idempotent elements of $k\{G\}$ lie in $k\{\pi_0(G)\}$, it follows that $(\ida\cap k\{\pi_0(G)\})=\ida$. Therefore $\sqrt{(\p\cap k\{\pi_0(G)\})}=\p$.
	
	If $\p$ is a $\s$-ideal, then $\p\cap k\{\pi_0(G)\}$ is a $\s$-ideal. Conversely, if $\p'\subseteq k\{\pi_0(G)\}$ is a $\s$-ideal, then $\sqrt{(\p')}\subseteq k\{G\}$ is a $\s$-ideal.
\end{proof}

To show that a $\s$-algebraic group has only finitely many invariant connected components it thus suffices to show that any $\s$-\'{e}tale $\s$-algebraic group has only finitely many invariant components. The latter statement holds indeed more generally:
\begin{lemma} \label{lemma: setalimpliesfinite}
	Let $R$ be a finitely $\s$-generated $k$-$\s$-algebra. If $R$ is $\s$-\'{e}tale, then $R$ has only finitely many prime $\s$-ideals.
\end{lemma}
\begin{proof}
	Since $R$ is $\s$-\'{e}tale, every prime ideal of $R$ is maximal and hence also minimal. If $\p$ is a prime $\s$-ideal of $R$, then $\p\subseteq\s^{-1}(\p)$ and therefore $\p=\s^{-1}(\p)$. So every prime $\s$-ideal of $R$ is reflexive. It follows from the Ritt-Raudenbush basis theorem (cf. Theorems~2.5.11 and 2.5.7 in \cite{Levin:difference}) that a finitely $\s$-generated $k$-$\s$-algebra has only finitely many minimal reflexive prime $\s$-ideals. Since every prime $\s$-ideal of $R$ is a minimal reflexive prime $\s$-ideal of $R$, this implies that $R$ has only finitely many prime $\s$-ideals.
\end{proof}
The main results of this section now follow from the above considerations:

\begin{theo} \label{theo: finiteness of scomponents}
	A $\s$-algebraic group has only finitely many invariant connected components.
\end{theo}
\begin{proof}
	Let $G$ be a $\s$-algebraic group. By Lemma \ref{lemma: correspondence components}, the invariant connected components of $G$ are in bijection with the invariant connected components of $\pi_0(G)$. By Lemma \ref{lemma: setalimpliesfinite} the $\s$-algebraic group $\pi_0(G)$ has only finitely many invariant connected components.
\end{proof}

\begin{proof}[Proof of Theorem \ref{theo: Hrushovskiconj}]
	By assumption $R=k\{G\}$ for a $\s$-algebraic group $G$. By Corollary~\ref{cor: minsprime is minimal} the set of minimal prime $\s$-ideals of $k\{G\}$ equals the set of minimal prime ideals of $k\{G\}$ that are $\s$-ideals. The latter set is finite by Theorem \ref{theo: finiteness of scomponents}.
\end{proof}

%
%

\section{Difference algebraic relations among solutions of linear differential equations}

Difference algebraic relations among solutions of differential equations are ubiquitous in the theory of special functions. A typical example is the recurrence relation $$xJ_{\alpha+2}(x)-2(\alpha+1)J_{\alpha+1}(x)+xJ_\alpha(x)=0$$
satisfied by the Bessel function $J_\alpha(x)$. A Galois theoretic framework to analyze these relations has been developed in \cite{DiVizioHardouinWibmer:DifferenceGaloisTheoryOfLinearDifferentialEquations} and \cite{DiVizioHardouinWibmer:DifferenceAlgebraicRel}. The Galois groups in this Galois theory are difference algebraic groups.

Combing this Galois theoretic approach with our Theorem \ref{theo: finiteness1}, we show that the difference ideal of all difference algebraic relations among the solutions of a linear differential equations is finitely generated.

\subsection{Difference algebras of finite presentation}

Before discussing differential equations and the relations among their solutions, we establish some preparatory results concerning difference algebras of finite presentation. It appears that the notion of a finitely presented difference algebra does not occur in the difference algebra literature. The results we shall need are a rather direct analog of standard results about algebras of finite presentation. However, a crucial difference to note is that, while a finitely generated algebra over a field is automatically finitely presented, a finitely $\s$-generated difference algebra over a difference field need not be finitely $\s$-presented.

As before, $k$ always denotes a $\s$-field.

\begin{defi} \label{defi: finite spresentation}
	A \ks-algebra $R$ is \emph{finitely $\s$-presented} if it is isomorphic to a quotient of a $\s$-polynomial ring over $k$ modulo a finitely $\s$-generated $\s$-ideal. That is, $R\simeq k\{y_1,\ldots,y_n\}/[f_1,\ldots,f_m]$.
\end{defi}

Note that, using Definition \ref{defi: finite spresentation}, Theorem \ref{theo: finiteness1} can be restated as: The coordinate ring $k\{G\}$ of a $\s$-algebraic group $G$ is finitely $\s$-presented.
%
%


\begin{lemma} \label{lemma: finite presentation}
	Let $\psi\colon k\{z_1,\ldots,z_\ell\}\to R$ be a surjective morphism of \ks-algebras from a $\s$-polynomial ring over $k$ onto $R$. If $R$ is finitely $\s$-presented, then the kernel of $\psi$ is a finitely $\s$-generated $\s$-ideal.
\end{lemma}
\begin{proof}
	We identify $R$ with $k\{y_1,\ldots,y_n\}/[f_1,\ldots,f_m]$. Choose $g_i\in k\{y\}$ such that $\psi(z_i)=\overline{g_i}$ for $i=1,\ldots,\ell$.
	Then $R$ is isomorphic to $k\{z,y\}/[z_i-g_i,f_j]$ via $\overline{y_j}\mapsto\overline{y_j}$. The inverse map is given
	by $\overline{z_i}\mapsto\overline{g_i},\ \overline{y_j}\mapsto\overline{y_j}$.
	
	Choose $h_j\in k\{z\}$ such that $\psi(h_j)=\overline{y_j}$ for $j=1,\ldots,n$ and consider the map $\varphi\colon k\{z,y\}\to k\{z\},\ z_i\mapsto z_i,\ y_j\mapsto h_j$.
	Then $\psi\circ \varphi$ is $k\{z,y\}\to R,\ \ z_i\mapsto\overline{g_i},\ y_j\mapsto\overline{y_j}$. As the kernel of $\psi\circ \varphi$ is $[z_i-g_i,f_j]$, it follows that the kernel of $\psi$ is $\s$-generated by $\varphi(z_i-g_i),\varphi(f_j)$.
\end{proof}

\begin{lemma} \label{lemma: finite spresentation and base extension}
	Let $R$ be a \ks-algebra and $k'/k$ a $\s$-field extension. Then $R$ is a finitely $\s$-presented \ks-algebra if and only if $R\otimes_k k'$ is a finitely $\s$-presented $k'$-$\s$-algebra.
\end{lemma}
\begin{proof}
	Clearly $R\otimes_k k'$ is finitely $\s$-presented if $R$ is. So let us assume that $R\otimes_k k'$ is a finitely $\s$-presented $k'$-$\s$-algebra. Then $R\otimes_k k'$ is finitely $\s$-generated over $k'$ and we may choose a finite set of $\s$-generators from $R\subseteq R\otimes_kk'$. Let $\psi\colon k\{y\}\to R$ be the corresponding morphism of \ks-algebras. Because the base change $\psi\otimes_k k'\colon k'\{y\}\to R\otimes_kk'$ is surjective, also $\psi$ must be surjective. 
	
	As $R\otimes_k k'$ is a finitely $\s$-presented $k'$-$\s$-algebra, it follows from Lemma \ref{lemma: finite presentation} that $\ker(\psi\otimes_kk')\subseteq R\otimes_kk'$ is finitely $\s$-generated. As $\ker(\psi\otimes_kk')=\ker(\psi)\otimes_kk'$ we can indeed find a finite set $F\subseteq\ker(\psi)$ that $\s$-generates $\ker(\psi\otimes_kk')$. Then 
	$$[F]\otimes_kk'=\ker(\psi\otimes_kk')=\ker(\psi)\otimes_kk'.$$
	Therefore $\ker(\psi)=[F]$ is finitely $\s$-generated and it follows that $R$ is finitely $\s$-presented.
\end{proof}

\subsection{Difference Galois theory of differential equations}

We briefly recall some of the basic notions from \cite{DiVizioHardouinWibmer:DifferenceGaloisTheoryOfLinearDifferentialEquations}. A \emph{$\de\s$-field} is a field $K$ together with a derivation $\de\colon K\to K$ and an endomorphism $\s\colon K\to K$ such that $\de$ and $\s$ commute up to a factor annulled by the derivation. We will always assume that $K$ is a $\de\s$-field of characteristic zero. A typical example, matching up with the example of the Bessel's function given above, is $K=\mathbb{C}(\alpha,x)$ with $\de=\frac{\partial}{\partial x}$ and $\s(f(\alpha,x))=f(\alpha+1,x)$.
Note that for a $\de\s$-field $K$, the field $k=K^\de=\{a\in K\ |\ \de(a)=0\}$ is a $\s$-field.

A \emph{$\s$-Picard-Vessiot extension} for $\de(y)=Ay$, where $A\in K^{n\times n}$, is a $\de\s$-field extension $L/K$ such that
\begin{itemize}
	\item there exists $Y\in\Gl_n(L)$ with $\de(Y)=AY$ such that $L$ is $\s$-generated by the entries of $Y$ as a $\s$-field extension of $K$ and
	\item $L^\de=K^\de$.
\end{itemize}	
Note that the $K$-$\s$-algebra $S=K\{Y,\frac{1}{\det(Y)}\}$, called a \emph{$\s$-Picard-Vessiot ring} for $\de(y)=Ay$, is also stable under the derivation $\de$. It encodes the difference algebraic relations among the solutions of $\de(y)=Ay$: If $Z$ is an $n\times n$-matrix of $\s$-indeterminates over $K$, then the kernel $\ida $ of $K\{Z,\frac{1}{\det(Z)}\}\to S,\ Z\mapsto Y$ is the difference ideal of all difference algebraic relations among the entries of $Y$ and $K\{Z,\frac{1}{\det(Z)}\}/\ida\simeq S$.

The $\s$-Galois group $G$ of $L/K$ is the functor from $\ksalg$ to $\groups$ that associates to every \ks-algebra $R$, the group of $\de\s$-automorphism of $S\otimes_k R$ over $K\otimes_k R$. (The derivation $\de$ on $R$ is assumed to be trivial.) It is shown in \cite[Prop. 2.5]{DiVizioHardouinWibmer:DifferenceGaloisTheoryOfLinearDifferentialEquations} that $G$ is a $\s$-algebraic group. It measures the difference algebraic relations among the entries of $Y$. We are now prepared to prove the main result of this section:

\begin{theo}
	Let $L/K$ be a $\s$-Picard-Vessiot extension for $\delta(y)=Ay$ and $Y\in\Gl_n(L)$ such that $\de(Y)=AY$. Then the kernel of $K\{Z,\frac{1}{\det(Z)}\}\to L,\ Z\mapsto Y$ is a finitely $\s$-generated $\s$-ideal. 
\end{theo}
\begin{proof}
	Let $z$ be a further $\s$-variable. If the kernel of $K\{Z,z\}\to L,\ Z\mapsto Y,\ z\mapsto \frac{1}{\det(Y)}$ is finitely $\s$-generated, also the kernel of $K\{Z,\frac{1}{\det(Z)}\}\to L,\ Z\mapsto Y$ is finitely $\s$-generated.
	Therefore, by Lemma \ref{lemma: finite presentation}, it suffices to show that the $\s$-Picard-Vessiot ring $S=K\{Y,\frac{1}{\det(Y)}\}$ is a finitely $\s$-presented $K$-$\s$-algebra. The algebraic recast of the torsor theorem (\cite[Lemma 2.7]{DiVizioHardouinWibmer:DifferenceGaloisTheoryOfLinearDifferentialEquations}) tells us that $S\otimes_KS\simeq S\otimes_k k\{G\}$, where $G$ is the $\s$-Galois group of $L/K$.
	
	This isomorphism extends to $L\otimes_KS\simeq L\otimes_k k\{G\}$. By Theorem \ref{theo: finiteness1} the coordinate ring $k\{G\}$ is a finitely $\s$-presented \ks-algebra. Therefore $L\otimes_k k\{G\}\simeq L\otimes_K S$ is a finitely $\s$-presented $L$-$\s$-algebra. It thus follows from Lemma \ref{lemma: finite spresentation and base extension} that $S$ is a finitely $\s$-presented $K$-$\s$-algebra.
\end{proof}


\bibliographystyle{alpha}
\bibliography{bibdata}
\end{document}